\documentclass[12pt]{article}  
\usepackage{amsmath}
\usepackage{amsfonts}
\usepackage{amssymb}
\usepackage{amsthm}
\usepackage{euscript}
\usepackage{ifthen}
\usepackage{xifthen}
\usepackage{xcolor}
\usepackage{fullpage}
\usepackage[english]{babel}
\usepackage{mathtools}
\usepackage{hyperref}
\usepackage{authblk}
\hypersetup{colorlinks = true,linkcolor=black,citecolor=black}

\usepackage{algorithm} 
\usepackage{algorithmicx}
\usepackage[noend]{algpseudocode}

\usepackage{array}

\usepackage{imakeidx}
\makeindex

\usepackage{color}

\usepackage{hyperref}
\hypersetup
{
	colorlinks,
	citecolor=black,
	filecolor=black,
	linkcolor=blue,
	urlcolor=blue
}
\hypersetup{linktocpage}
\newtheorem{theorem}{Theorem}[section]
\newtheorem{lemma}[theorem]{Lemma}

\numberwithin{equation}{section}

\theoremstyle{definition}
 
\theoremstyle{remark}
\newtheorem{remark}[theorem]{Remark}

\newcommand{\brac}[1]{\left(#1\right)}

\newcommand{\brab}[1]{\left\{#1\right\}}

\newcommand{\bj}{{\boldsymbol{j}}}
\newcommand{\bk}{{\boldsymbol{k}}}
\newcommand{\bell}{{\boldsymbol{\ell}}}

\newcommand{\bm}{{\boldsymbol{m}}}
\newcommand{\bn}{{\boldsymbol{n}}}

\newcommand{\bs}{{\boldsymbol{s}}}

\newcommand{\bx}{{\boldsymbol{x}}}

\newcommand{\by}{{\boldsymbol{y}}}
\newcommand{\bX}{{\boldsymbol{X}}}

\newcommand{\bW}{{\boldsymbol{W}}}

\newcommand{\bone}{{\boldsymbol{1}}}

\newcommand{\rd}{{\rm d}} 
\newcommand{\Int}{{\rm Int}}

\def\ZZd{{\mathbb Z}^d}

\def\ZZ{{\mathbb Z}}

\def\RR{{\mathbb R}}
\def\RR{{\mathbb R}^d}

\def\NN{{\mathbb N}}
\def\NNd{{\NN}^d}

\def\NN{{\mathbb N}}
\def\RR{{\mathbb R}}

\def\NNd{{\mathbb N}^d}
\def\RR{{\mathbb R}^d}

\def\ZZd{{\mathbb Z}^d}

\def\Ii{{\mathcal I}}

\def\Pp{{\mathcal P}}

\def\Ss{{\mathcal S}}

\def\ZZ{{\mathbb Z}}
\def\NN{{\mathbb N}}
\def\RR{{\mathbb R}}

\def\NNd{{\mathbb N}^d}
\def\RRd{{\mathbb R}^d}

\def\supp{\operatorname{supp}}

\def\Wap{W^r_p}

\newcommand{\norm}[2]{\left\|{#1}\right\|_{#2}}

\title{\sffamily Weighted 
	approximate sampling 
	recovery and  integration based on  B-spline interpolation and quasi-interpolation} 

\author[a]{Dinh D\~ung}
\affil[a]{Information Technology Institute, Vietnam National University, Hanoi
	\protect\\
	144 Xuan Thuy, Cau Giay, Hanoi, Vietnam
	\protect\\
	Email: dinhzung@gmail.com}

\date{\today}
\begin{document}
\maketitle

\begin{abstract}
We propose novel methods  for  approximate  sampling recovery and integration of functions in the Freud-weighted Sobolev space $W^r_{p,w}(\mathbb{R})$.  The approximation error of sampling recovery is measured in the norm of  the Freud-weighted Lebesgue space $L_{q,w}(\mathbb{R})$. Namely, we construct   equidistant,   compact-supported  B-spline quasi-interpolation and interpolation sampling algorithms $Q_{\rho,m}$ and  $P_{\rho,m}$  which are asymptotically optimal in terms of the sampling $n$-widths  $\varrho_n(\boldsymbol{W}^r_{p,w}(\mathbb{R}),  L_{q,w}(\mathbb{R}))$ for every pair $p,q \in [1,\infty]$, and prove the exact convergence rate of these sampling $n$-widths, where $\boldsymbol{W}^r_{p,w}(\mathbb{R})$ denotes the unit ball in $W^r_{p,w}(\mathbb{R})$. The algorithms $Q_{\rho,m}$ and  $P_{\rho,m}$ are based on truncated scaled  B-spline quasi-interpolation and interpolation, respectively. We also prove  the asymptotical optimality and  exact convergence rate of the  equidistant quadratures generated from  $Q_{\rho,m}$ and  $P_{\rho,m}$,  for Freud-weighted numerical integration of functions in $W^r_{p,w}(\mathbb{R})$.

	
	
\end{abstract}

\section{Introduction}
\label{Introduction}

The aim of this paper is to construct linear sampling algorithms based on equidistant,  compact-support B-spline interpolation and quasi-interpolation, for   approximate recovery of  univariate functions in the weighted  Sobolev space $W^r_{p,w}(\mathbb{R})$ of smoothness $r \in \mathbb{N}$. The approximate recovery of   functions is based on a finite number of their sampled values. The approximation error is measured in the norm of the weighted Lebesgue space $L_{q,w}(\mathbb{R})$.  Here,  $w$ is a  Freud  weight, and the parameters  $p,q \in [1,\infty]$ may take different values. The optimality of sampling algorithms is investigated  in terms of sampling $n$-widths of the unit ball $\boldsymbol{W}^r_{p,w}(\mathbb{R})$  in this space.  We are also concerned with the numerical integration and optimal quadrature based on B-spline interpolation and quasi-interpolation for functions in $W^r_{p,w}(\mathbb{R})$.

We begin with definitions of weighted  function spaces.  
Let \begin{equation} \nonumber
	w(\bx):= w_{\lambda,a,b}(\bx) := \bigotimes_{i=1}^d w(x_i), \ \ \bx \in \RRd,
\end{equation}
be the tensor product of $d$ copies of a  univariate Freud  weight of the form
\begin{equation} \label{w(x)}
	w(x):=  w_{\lambda,a,b}(x) := \exp \brac{- a|x|^\lambda + b},
	 \ \ \lambda > 1,  \ a >0, \  b \in \RR.
\end{equation} 
The most important parameter in the weight $w$ is $\lambda$. The parameter $b$ which produces only a positive constant in the weight $w$  is introduced for a certain normalization, for instance, for the standard Gaussian weight which is one of the most important weights. 	In what follows, for simplicity of presentation, 
without  loss of generality we  assume $b=0$, and
fix the weight $w$ and hence the  parameters $\lambda, a$.

Let  $1\leq q<\infty$ and $\Omega$ be a Lebesgue measurable  subset of $\RRd$. 
We denote by  $L_{q,w}(\Omega)$ the weighted Lebesgue space  of all measurable functions $f$ on $\Omega$ such that the norm
\begin{align} \label{L-Omega}
\|f\|_{L_{q,w}(\Omega)} : = 
\bigg( \int_\Omega |f(\bx)w(\bx)|^q  \rd \bx\bigg)^{1/q}
\end{align}
is finite.
For $q=\infty$, we define  the space $L_{\infty,w}(\Omega):=C_w(\Omega)$ of all continuous functions on $\Omega$ such that
the norm 
\begin{equation*}
	\norm{f}{L_{\infty,w}(\Omega)}
	:=  
	\sup_{x\in \Omega}| f(\bx)w(\bx)|
\end{equation*}
is finite.
 For $r \in \NN$ and $1 \le p \le \infty$, the weighted  isotropic Sobolev space $W^{r,\operatorname{iso}}_{p,w}(\Omega)$   is defined as the normed space of all functions $f\in L_{p,w}(\Omega)$ such that the weak  partial derivative $D^{\bk} f$ belongs to $L_{p,w}(\Omega)$ for  every $\bk \in \NNd_0$ with $k_1+\cdots +k_d \le r$. 
Here, the letters 'iso' in the suffix is to distinct the notation  for  weighted  isotropic Sobolev space from the notation for mixed-smoothness Sobolev space $W^r_{p,w}(\Omega)$ which has already been employed in the author's prior works.
 For $d=1$ this means that the  derivative $f^{(r-1)}$ is absolute continuous and $f^{(r)}  \in L_{p,w}(\Omega)$. In this case, the letters 'iso' are omitted.
 The norm of a  function $f$ in this space is defined by
\begin{align} \label{W-Omega}
	\|f\|_{W^{r,\operatorname{iso}}_{p,w}(\Omega)}: = \Bigg(\sum_{k_1+\cdots +k_d \le r} \|D^{\bk} f\|_{L_{p,w}(\Omega)}^p\Bigg)^{1/p}.
\end{align}

For the standard $d$-dimensional Gaussian measure  $\gamma$ with the density function 
$$
v_{\operatorname{g}}(\bx): = (2\pi)^{-d/2}\exp (- |\bx|_2^2/2),
$$
consider the classical  spaces $L_p(\Omega;\gamma)$ and $W^{r,\operatorname{iso}}_{p,w}(\Omega; \gamma)$ 
which are used in many theoretical and applied problems.  The norm in \eqref{L-Omega} for these spaces takes the form
$$ \nonumber
\|f\|_{L_p(\Omega; \gamma)} : = 
\bigg( \int_\Omega |f(\bx)|^p \gamma(\rd \bx)\bigg)^{1/p}
=
\bigg( \int_\Omega |f(\bx)\brac{v_{\operatorname{g}}}^{1/p}(\bx)|^p  \rd \bx\bigg)^{1/p}.
$$
Thus, the spaces $L_p(\Omega;\gamma)$ and $W^{r,\operatorname{iso}}_{p,w}(\Omega; \gamma)$ with the Gaussian measure can be seen as  the Gaussian-weighted spaces   $L_{p,w}(\Omega)$ and $W^{r,\operatorname{iso}}_{p,w}(\Omega)$ with $w:= \brac{v_{\operatorname{g}}}^{1/p}$
for  a fixed $1 \le p < \infty$.  

The spaces $L_p(\Omega;\gamma)$ and $W^{r,\operatorname{iso}}_{p,w}(\Omega; \gamma)$ with the standard Gaussian measure can be generalized for any positive measure.
Let  $\Omega \subset \RRd$ be a Lebesgue measurable set. Let $v$ be a nonzero nonnegative Lebesgue measurable  function on $\Omega$. Denote by $\mu_v$ the measure on $\Omega$ defined via the density function $v$, i.e., for every Lebesgue measurable set $A \subset \Omega$,
$$
\mu_v(A) := \int_A v(\bx) \rd \bx.
$$
For $1 \le p < \infty$, let $L_p(\Omega;\mu_v)$ be the space  with measure $\mu_v$ of all Lebesgue measurable functions $f$ on $\Omega$ such that the norm
\begin{align} \nonumber
	\|f\|_{L_p(\Omega;\mu_v)} : = 
	\bigg( \int_\Omega |f(\bx)|^p  \mu_v(\rd \bx)\bigg)^{1/p} 
	=
	\bigg( \int_\Omega |f(\bx)|^p v(\bx) \rd \bx\bigg)^{1/p} 
\end{align}
is finite.
For $r \in \NN$, the Sobolev spaces $W^{r,\operatorname{iso}}(\Omega;\mu_v)$ with measure $\mu_v$, and  the classical Sobolev space $W^{r,\operatorname{iso}}_{p,w}(\Omega)$ are  defined in the same way as  in\eqref{W-Omega}  by replacing $L_{p,w}(\Omega)$ with $L_p(\Omega; \mu)$ and $L_p(\Omega)$, respectively. 

Let us formulate a setting of  optimal linear sampling recovery problem. Let  $X$ be  a normed space  of functions on $\Omega$. Given sample points $\bx_1,\ldots,\bx_k \in \Omega$, we consider the approximate recovery  of a continuous function $f$ on $\Omega$  from their values $f(\bx_1),\ldots, f(\bx_k)$ by a linear sampling algorithm (operator) $S_k$
on $\Omega$ of the form
\begin{equation} \label{S_k}
	S_kf: = \sum_{i=1}^k  f(\bx_i) \phi_i, 
\end{equation}
where  $\phi_1,\ldots,\phi_k$ are given  functions on $\Omega$.  For convenience, we allow that some of the sample points $\bx_i$ may coincide.  The approximation error is measured by the norm 
$\|f - S_k f\|_X$.  Denote by $\Ss_n$  the family of all linear sampling algorithms $S_k$ of the form \eqref{S_k} with $k \le n$.
Let $F \subset X$ be a set of continuous functions on $\Omega$.   To study the optimality  of linear  sampling algorithms  from $\Ss_n$ for  $F$ and their convergence rates we use  the  (linear) sampling $n$-width
\begin{equation} \label{rho_n}
\varrho_n(F, X) :=\inf_{S_n \in \Ss_n} \ \sup_{f\in F} 
\|f - S_n f\|_X.
\end{equation}

For numerical integration, we are interested  in approximation of   the weighted integral 
\begin{equation} \nonumber
	\int_{\Omega} f(\bx) w(\bx) \, \rd \bx 
\end{equation}
for functions $f$ lying  in the space 
$W^{r,\operatorname{iso}}_{p,w}(\Omega)$ for $1 \le p \le \infty$.
To approximate them we use quadratures  (quadrature operators) $I_k$  of the form
\begin{equation} \label{I_kf}
	I_kf: = \sum_{i=1}^k \lambda_i f(\bx_i), 
\end{equation}
where $\bx_1,\ldots,\bx_k \in \Omega$  are  the integration nodes and $\lambda_1,\ldots,\lambda_k$ the integration weights.  For convenience, we assume that some of the integration nodes $\bx_i$ may coincide. 
Notice that every sampling algorithm $S_k \in \Ss_n$ generates in a natural way a quadrature $I_k \in \Ii_n$ by the formula
\begin{equation} \label{I_kf-generated}
	I_kf = \int_{\Omega} S_kf(\bx) w(\bx) \, \rd \bx  = \sum_{i=1}^k \lambda_i f(\bx_i) 
\end{equation}
with the integration weights 
$$
\lambda_i:= \int_{\Omega} \phi_i(\bx) w(\bx)\rd \bx.
$$ 
Let $F$ be a set of continuous functions on $\Omega$.  Denote by $\Ii_n$  the family of all quadratures $I_k$ of the form \eqref{I_kf} with $k \le n$. The optimality  of  quadratures from $\Ii_n$ for  $f \in F$  is measured by 
\begin{equation} \nonumber
	\Int_n(F) :=\inf_{I_n \in \Ii_n} \ \sup_{f\in F} 
	\bigg|\int_{\Omega} f(\bx) w(\bx) \, \rd \bx - I_nf\bigg|.
\end{equation}

In the present paper, we focus our attention mostly on the sampling recovery and numerical integration for functions on $\RRd$ in the one-dimensional case when $d =1$ and shortly consider the multidimensional case when $d >1$.  

Sampling recovery and numerical integration are ones of basic problems in approximation theory and numerical analysis. The number of papers devoted to these problems is too large to
	mention all of them.
 We refer the reader to \cite{DTU18B,NoWo08,NoWo10,Tem18B}
for detailed surveys and  bibliography.
B-spline quasi-interpolations possess good local and approximation properties (see \cite{Chui92,deBHR1993,DeLo93B}). They were used for unweighted sampling recovery and numerical integration \cite{Dung2018,Dung11a,Dung16,Tri10B} (see also \cite{DD2023-survey,DTU18B} for survey and bibliography). In these papers, the authors constructed efficient sampling algorithms and quadratures based on B-spline quasi-interpolations, for approximate recovery and numerical integration of functions in Sobolev and Besov spaces, and prove their convergence rates. The optimality was investigated in terms of  the sampling $n$-widths $\varrho_n(F,X)$ and the quantity of optimal integration  $\Int_n(F)$ over the unit ball in these spaces. There have been a large number of papers devoted to  Gaussian- or more general Freud-weighted  interpolation and sampling recovery \cite{DD2024,DK2022,GHHR2022,GHRR2024,JMN2021,Lubi82,MN2010,MV2007,OR2006,SK2024,Sza1997}, quadrature and numerical integration  \cite{DD2023,DK2022,DM2003,DILP18,GHHR2022,GHRR2024,GKS2024,IL2015,KSG2023,EG2022,MO2004}. 

The present paper is also related to Freud-weighted polynomial approximation, in particular, Freud-weighted polynomial interpolations and quadratures.  We refer the reader to the books and monographs \cite{JMN2021,Lu07B,Mha1996B} for surveys and bibliographies on this research direction. The Freud-weighted Lagrange polynomial interpolation on $\RR$ and relevant Gaussian quadrature based on the zeros of the orthonormal polynomials with respect to the weight $w^2$ is not efficient to approximate functions in $C_w(\RR)$ and their weighted integrals \cite{Sza1997}, \cite[Proposition 1]{DM2003}. To overcome such problems, there were several suggestions of  section of the truncated sequence of these zeros and the Mhaskar-Rakhmanov-Saff points $\pm a_m$  for construction 
 of polynomial interpolation \cite{MN2010,MV2007,OR2006,Sza1997} and quadrature \cite{DM2003,MO2004} for efficient approximation.   The optimality of the polynomial interpolation and quadrature considered in \cite{MN2010} and  \cite{DM2003}, has been confirmed in \cite{DD2024} and \cite{DD2023}, respectively,  for some particular cases.

In previous works on one-dimensional Gaussian- and Freud-weighted interpolation and quadrature, the authors used the zeros of the orthonormal polynomials with respect to the weight $w^2$ or a part or a modification  of them as interpolation and quadrature nodes (cf. \cite{DD2023,DD2024,DM2003,EG2022,JMN2021,KSG2023,Lubi82,MN2010,MO2004,MV2007,OR2006,SK2024,Sza1997}). 
This requires to compute with a certain accuracy  the values of these non-equidistant zeros and  of functions at these points. Moreover, the methods employed there do not give optimal sampling recovery algorithms and quadratures for example,  for functions from the Sobolev space $W^r_{p,w}(\RR)$ in the important cases when $p = 1, \infty$.  In the present paper,  we overcome these disadvantages by proposing novel methods for construction of B-spline interpolation and quasi-interpolation and quadrature for optimal weighted sampling recovery and numerical integration of smooth functions using equidistant sample and quadrature nodes which are much simpler and easier for computation, since these nodes  and the employed B-splines can be easily and explicitly constructed, and the practical B-spline computation is well-known (for detail, see Remark~\ref{remark1}). Moreover, B-splines are a powerful tool in both theoretical and applied disciplines, including approximation theory and computational mathematics. For surveys of the topic and an extensive bibliography, see the references \cite{Chui92,DTU18B,deBoor1977,deBoor1978,deBHR1993}.  

Let  $p,q \in [1,\infty]$ be any pair. We construct compact-supported  equidistant quasi-interpolation and interpolation sampling algorithms $Q_{\rho,m}$ and  $P_{\rho,m}$ (see \eqref{Q_{rho,m}:=} and \eqref{P_{rho,m}:=}, respectively,  for definition) which are asymptotically optimal in terms of  $\varrho_n(\bW^r_{p,w}(\RR),  L_{q,w}(\RR))$. These algorithms are based on  truncated scaled cardinal B-spline quasi-interpolation and relevant B-spline interpolation of even order $2\ell$, and constructed from $2(m +\ell + j_0) - 1$  sample function values at certain equidistant points, where $j_0$ is a constant nonnegative  integer associated with B-spline quasi-interpolation. We prove that $I^Q_{\rho,m}$ and  $I^P_{\rho,m}$, the equidistant quadratures   generated from  $Q_{\rho,m}$ and  $P_{\rho,m}$ by formula \eqref{I_kf-generated}, are asymptotically optimal for 
 $\Int_n\big(\bW^r_{p,w}(\RR)\big)$. We compute the exact convergence rates of 
 $\varrho_n(\bW^r_{p,w}(\RR),  L_{q,w}(\RR))$ and  $\Int_n\big(\bW^r_{p,w}(\RR)\big)$. 
 We also prove some Marcinkiewicz-  Nikol'skii- and Bernstein-type inequalities for scaled cardinal B-splines, which play a basic role in establishing the optimality of the algorithms $Q_{\rho,m}$ and $P_{\rho,m}$. In particular, these results are true for the Gaussian-weighted spaces $L_p(\RR;\gamma)$ and $\Wap(\RR; \gamma)$. 
 
We shortly describe the main results  of our paper.  Throughout this paper, for given $p, q \in [1, \infty]$ and the parameter $\lambda > 1$ in the definition \eqref{w(x)} of the  univariate weight $w$,  we make use of  the notations
$$
r_\lambda:= r(1 - 1/\lambda);
$$ 
\begin{equation}\nonumber
	\delta_{\lambda,p,q} 
	:=
	\begin{cases}
		(1-1/\lambda)(1/p - 1/q)&  \ \ \text{if} \ \ p \le q, \\
		(1/\lambda)(1/q - 1/p)&  \ \ \text{if} \ \ p > q;
	\end{cases}	 
\end{equation}
(with the convention $1/\infty := 0$)  and 
$$
r_{\lambda,p,q}:= r_\lambda - 	\delta_{\lambda,p,q}.
$$

	Let  $1\le p,q \le \infty$ and $r_{\lambda,p,q} >0$.   For any $n \in \NN$, let $m(n)$ be the largest integer such that $2(m +\ell + j_0) - 1 \le n$.  Let the sampling operator $S_n \in \Ss_n$ be either $Q_{\rho,m(n)}$ or  $P_{\rho,m(n)}$. Then  $S_n$ is asymptotically optimal for the sampling $n$-widths $\varrho_n\big(\bW^r_{p,w}(\RR),  L_{q,w}(\RR)\big)$, and
\begin{equation}\label{rho_n}
	\varrho_n\big(\bW^r_{p,w}(\RR),  L_{q,w}(\RR)\big) 
	\asymp 
	\sup_{f\in \bW^r_{p,w}(\RR)} \big\|f - S_nf\big\|_{ L_{q,w}(\RR)} 
	\asymp
	n^{- r_{\lambda,p,q}},
\end{equation}
(for detail, see Theorem~\ref{thm:S_n}).

 Since the function spaces $L_p(\RR;\mu_w)$ and $W^r_p(\RR;\mu_w)$ with the measure $\mu_w$ coincide with $L_{p,w^{1/p}}(\RR)$  and $W^r_{p,w^{1/p}}(\RR)$ for  $1\le p < \infty$, respectively, from \eqref{rho_n} it follows that  the sampling algorithm $S_n$ is asymptotically optimal for the sampling $n$-widths 
 $\varrho_n(\bW^r_p(\RR;\mu_w), L_p(\RR;\mu_w))$ for  $1\le p < \infty$, $r_\lambda >0$, and
\begin{equation}\nonumber
\varrho_n(\bW^r_p(\RR;\mu_w), L_p(\RR;\mu_w))
	\asymp 
	\sup_{f\in \bW^r_p(\RR;\mu_w)} \big\|f - S_nf\big\|_{ L_p(\RR;\mu_w)} 
	\asymp
	n^{- r_\lambda};
\end{equation}
and, in particular, $S_n$ is asymptotically optimal for Gaussian-weighted sampling recovery in terms of the sampling $n$-widths 
$\varrho_n(\bW^r_p(\RR;\gamma), L_p(\RR;\gamma))$ for $r >0$, and
\begin{equation}\nonumber
	\varrho_n(\bW^r_p(\RR;\gamma), L_p(\RR;\gamma))
	\asymp 
	\sup_{f\in \bW^r_p(\RR;\gamma)} \big\|f - S_nf\big\|_{ L_p(\RR;\gamma)} 
	\asymp
	n^{- r/2}.
\end{equation}

	Let  $1\le p\le \infty$ and $r_\lambda - (1/\lambda)(1 - 1/p) >0$.   For any $n \in \NN$, let $m(n)$ be the largest integer such that $2(m + \ell +j_0)  - 1 \le n$.  Let the quadrature $I_n \in \Ii_n$ be either   $I^Q_{\rho,m(n)}$ or  $I^P_{\rho,m(n)}$ generated by the formula \eqref{I_kf-generated} from  $Q_{\rho,m}$ and  $P_{\rho,m}$, respectively.  Then  $I_n$ is asymptotically optimal in terms of  $\Int_n\big(\bW^r_{p,w}(\RR)\big)$, and
\begin{equation}\label{I_n}
	\Int_n\big(\bW^r_{p,w}(\RR)\big) 
	\asymp 
	\sup_{f\in \bW^r_{p,w}(\RR)} \left|\int_{\RR} f(x) w(x) \, \rd x - I_nf \right|
	\asymp
	n^{- r_\lambda + (1/\lambda)(1 - 1/p)} \ \  \forall n \in \NN,
\end{equation}
(for detail, see Theorem~\ref{thm:I_n}).

Analogously, \eqref{I_n} yields that for the function spaces $L_p(\RR;\mu_w)$ and $W^r_p(\RR;\mu_w)$ with the measure $\mu_w$,  the quadrature $I_n$ is asymptotically optimal in terms of
$\Int_n(\bW^r_1(\RR;\mu_w))$ and of 
$\Int_n(\bW^r_1(\RR;\gamma))$ for $r > 0$. Moreover, 
\begin{equation}\nonumber
\Int_n(\bW^r_1(\RR;\mu_w))
	\asymp 
	\sup_{f\in \bW^r_1(\RR;\mu_w)} \left|\int_{\RR} f(x)  \, \rd \mu_w(x) - I_nf \right|
	\asymp
	n^{- r_\lambda},
\end{equation}
and, in particular, 
\begin{equation}\nonumber
	\Int_n(\bW^r_1(\RR;\gamma))
	\asymp 
	\sup_{f\in \bW^r_1(\RR;\gamma)} \left|\int_{\RR} f(x)  \, \rd \gamma(x) - I_nf \right|
	\asymp
	n^{- r/2}.
\end{equation}

Recently, a sequence of works by the author of this paper and his collaborator on weighted sampling recovery and numerical integration over $\RR$ and $\RRd$
has appeared and bears directly on the themes of the present study. Here, we offer comments on the results of those papers, with a focus on the one-dimensional case 
$\RR$, and contrast them with the main findings of the present work.

In the paper \cite{DK2022}, we established the exact convergence rate 
of  $\varrho_n((\bW^r_p(\RR;\gamma), L_q(\RR;\gamma))$  for $1\le q < p \le\infty $ and $r \ge 2$, and the exact convergence rate of $\Int_n(\bW^r_p(\RR;\gamma))$, respectively,  for 
$1<  p <\infty $ and $r \ge 1$.  The exact convergence rates are achieved by  sampling and quadrature algorithms that assemble asymptotically optimal sampling  and quadrature algorithms for the related Sobolev spaces on the unit interval transferred to the  integer-shifted interval.
	In the recent paper \cite{DD2025a}, we have extended these results to a measure $\mu_w$ of density function $w$  as in \eqref{w(x)} with arbitrary $\lambda >0$.

In the  work \cite{DD2023}, we proved  the exact convergence rate of  $\Int_n(\bW^r_{1,w}(\RR))$. In the work \cite{DD2024}, we proved  the exact convergence rate of $\varrho_n\big(\bW^r_{p,w}(\RR),  L_{q,w}(\RR)\big)$ for $1 <p< \infty$ and $1\le q \le \infty$. 
  The exact convergence rates are achieved by  generalized methods  of  truncated Lagrange interpolation and Gaussian quadratures from \cite{MN2010} and \cite{DM2003}, respectively.

In \cite{DD2025b}, we established  in a non-constructive manner, the exact convergence rates of $\varrho_n\big(\bW^r_p(\RR; \mu_w), L_q(\RR; \mu_w)\big)$ for $1 \le q \le 2 < p \le \infty$ and  of $\varrho_n(\bW^r_2(\RR;\mu_w),  L_q(\RR;\mu_w))$  for $1 \le q \le 2$. The argument for the first result hinges on the exact convergence rates  of the Kolmogorov $n$-widths 
$d_n\big(\bW^r_p(\RR; \mu_w), L_q(\RR; \mu_w)\big)$ and a recent result on sampling $n$-widths in  \cite[Corollary~4]{DKU22}.  A key role playing in the proof of the second result are a  RKHS structure of the space $W^r_2(\RR;\mu_w)$, which is derived from some old  results \cite{Bon1984,BP1984,Freud1976} on properties of the relevant orthonormal polynomials, and the recent finding \cite[Corollary 2]{DKU22} on sampling $n$-widths.

Notice that in the papers referenced above,  two distinct settings of optimal weighted sampling recovery and numerical integration are considered: \ (i)  A weighted setting via the quantities $\varrho_n(\boldsymbol{W}^r_{p,w}(\mathbb{R}), L_{q,w}(\RR))$ and $\Int_n(\boldsymbol{W}^r_{p,w}(\mathbb{R}))$, and \ (ii) a measure-based setting via the quantities 
$\varrho_n\big(\boldsymbol{W}^r_{p}(\mathbb{R}; \mu_w), L_{q}(\mathbb{R}; \mu_w)\big)$ 
$\Int_n(\boldsymbol{W}^r_{p}(\mathbb{R}; \mu_w))$.
Setting (i) comes from the classical theory of weighted approximation (for knowledge and bibliography see, e.g., \cite{Mha1996B}, \cite{Lu07B},  \cite{JMN2021}). Setting (ii) is related to many theoretical and applied topics,  especially  to Gaussian measure $\gamma$ and other probability measures $\mu_w$. Our paper concentrates on setting (i). 
The results  for setting (ii) in the particular case $1 \le p=q <\infty$ follow as consequences from the results established in setting (i).  A careful examination of the cited works shows that, in general, settings (i) and (ii) yield substantially different approximation results, except the case $1 \le p=q <\infty$ for sampling recovery, and the case $p=1$ for numerical integration, when they are coincide, up to a re-notation.

Finally, we emphasize that the approaches developed in the cited papers are distinct from, and not reducible to, the novel methods employed in this work. Our methods are based on equidistant nodes combined with B-spline interpolation and quasi-interpolation. This constitutes the first fundamental contribution of our paper.
		As noted above, another significant contribution of this paper is that our results establish the  convergence rates for two fundamental problems in weighted spaces: 
	optimal sampling recovery in $L_{q,w}(\RR)$ and optimal quadrature of  functions from $\bW^r_{p,w}(\RR)$. These results   hold for all the pair $p,q \in [1,\infty]$, and, importantly, include the cases  $p = 1, \infty$ which were not treated  in prior works.	

It turns out that all the results of the one-dimensional case ($d = 1$)
	can be generalized to the multidimensional case ($d > 1$).  
It is interesting to generalize and extend these results to multivariate functions having a mixed smoothness. This problem will be devoted in an upcoming paper.

The paper is organized as follows. In Section \ref{B-spline sampling recovery}, we construct  truncated compact-supported B-spline  quasi-interpolation and interpolation, respectively,  algorithms and prove the error estimate of the approximation by them. Section \ref{Optimality of sampling algorithms} is devoted to the problem of optimality of sampling algorithms in terms of sampling $n$-widths.
In Subsection~\ref{Weighted B-spline inequalities}, we prove some Marcinkiewicz-  Nikol'skii- and Bernstein-type inequalities for scaled cardinal B-splines on $\RR$, which will be used for establishing the optimality of the B-spline quasi-interpolation and interpolation algorithms in the next subsection.
In Subsection \ref{Optimality}, we prove the optimality of B-spline  quasi-interpolation and interpolation algorithms in terms of the sampling $n$-widths $\varrho_n(\bW^r_{p,w}(\RR),  L_{q,w}(\RR))$, and compute the exact convergence rate of these sampling $n$-widths. 
In Section \ref{Numerical integration}, we prove that  the equidistant quadratures   generated from  the truncated B-spline  quasi-interpolation and interpolation algorithms, are asymptotically optimal in terms 
of $\Int_n\big(\bW^r_{p,w}(\RR)\big)$, and compute the exact convergence rate of 
$\Int_n\big(\bW^r_{p,w}(\RR)\big)$.
In Section \ref{High dimensional generalization}, we formulate a generalization of all the results in the previous sections to  the multidimensional case when $d > 1$.

\medskip
\noindent
{\bf Notation.} 
Denote   $\bx=:\brac{x_1,...,x_d}$ for $\bx \in \RRd$. 
For $\bx, \by \in \RRd$, the inequality $\bx \le \by$ ($\bx < \by$) means $x_i \le y_i$ ($x_i < y_i$) for every $i=1,...,d$.
We use letters $C$  and $K$ to denote general 
positive constants which may take different values. For the quantities $A_n(f,\bk)$ and $B_n(f,\bk)$ depending on 
$n \in \NN$, $f \in W$, $\bk \in J \subset \ZZd$,  
we write  $A_n(f,\bk) \ll B_n(f,\bk)$ \  $ \forall f \in W$, \ $\forall \bk \in J$ ($n \in \NN$ is specially dropped),  
if there exists some constant $C >0$ independent of $n,f,\bk$  such that 
$A_n(f,\bk) \le CB_n(f,\bk)$ for all $n \in \NN$,  $f \in W$, $\bk \in \ZZd$ (the notation $A_n(f,\bk) \gg B_n(f,\bk)$ has the  opposite meaning), and  
$A_n(f,\bk) \asymp B_n(f,\bk)$ if $S_n(f,\bk) \ll B_n(f,\bk)$
and $B_n(f,\bk) \ll S_n(f,\bk)$.  Denote by $|G|$ the cardinality of the set $G$. 
For a Banach space $X$, denote by the boldface $\bX$ the unit ball in $X$.

\section{B-spline sampling recovery}
\label{B-spline sampling recovery}

In this section, we construct  truncated equidistant,  compact-supported B-spline  quasi-interpolation  and interpolation algorithms and prove bounds of the error of the approximation by them.

	\subsection{B-spline quasi-interpolation}
\label{B-spline quasi-interpolation}

Recall that through this paper, for the weight $w$ defined as in \eqref{w(x)},  the parameters $\lambda > 1$ and $a >0$  are fixed, and $b=0$.  For $m \in \NN$, let $a_m$  be the Mhaskar-Rakhmanov-Saff number defined by
\begin{equation}\nonumber
	a_m :=\nu_\lambda m^{1/\lambda}, \ \ 
\nu_\lambda:= 	\brac{2^{\lambda - 1} \Gamma(\lambda)^{-1} \Gamma(\lambda/2)^2}^{1/\lambda},
\end{equation}
where $\Gamma$ is the gamma function. 
The number $a_m$ is relevant to convergence rates of weighted polynomial approximation (see, e.g., \cite{Mha1996B,Lu07B}).
We will need the following auxiliary result. 

\begin{lemma} \label{lemma:|f|_{L_{q,w}(RR setminus [-rho a_m, rho a_m])}}
	Let $1 \le p,q \le \infty$ and $0 < \rho <1$.	Then
	\begin{equation}\nonumber
		\|f\|_{L_{q,w}(\RR \setminus [-\rho a_m, \rho a_m])} 
		\le
		C  m^{- r_{\lambda,p,q}} \|f \|_{W^r_{p,w}(\RR)} \ \ \forall f \in W^r_{p,w}(\RR),   \ \forall m \in \NN, 
	\end{equation}
	where $C$ is a positive constant  independent of $m$ and $f$.
\end{lemma}

\begin{proof}
Denote by $\Pp_m$ the space of polynomials of degree at most $m$. For $f \in L_{p,w}(\RR)$, we define 	
\begin{equation}\nonumber
	E_m(f)_{p,w}:= \inf_{\varphi \in \Pp_m}	\|f - \varphi\|_{L_{p,w}(\RR)} 
\end{equation}
as the quantity of best weighted approximation of $f$ by polynomials of degree at most $m$.
For the following inequality  see \cite[(3.4)]{MV2007}.
With
$
	M(m)
	:= 	\left\lfloor  \frac{\rho}{\rho + 1} m\right\rfloor,
$
we have
\begin{equation} \nonumber
	\|f\|_{L_{q,w}(\RR \setminus [-\rho a_m, \rho a_m])} 
	\le
	C \brac{E_{M(m)}(f)_{q,w} + e^{-Km}\|f \|_{L_{q,w}(\RR)}} \ \ \forall f \in L_{q,w}(\RR), \ \forall m \in \NN,  
\end{equation}
where $C$ and $K$ are positive constants independent of $m$ and $f$.
There holds the inequality \cite[Theorem 2.3]{DD2024a}
\begin{equation}\nonumber
	E_m(f)_{q,w} 
	\le 
	C m^{-r_{\lambda,p,q}}  \|f \|_{W^r_{p,w}(\RR)}\ \ \forall f \in W^r_{p,w}(\RR),   \  \forall m \in \NN,  
\end{equation}
	where $C$  is a positive constant independent of $m$, $\varphi$.
		
	Let 	$ f  \in W^r_{p,w}(\RR)$ and $\forall m \in \NN$.  From the last inequalities we deduce 
	\begin{equation} \nonumber
		\begin{aligned}
	\|f\|_{L_{q,w}(\RR \setminus [-\rho a_m, \rho a_m])} 
&\ll
E_{M(m)}(f)_{q,w} + e^{-Km}\|f \|_{L_{q,w}(\RR)}
\\&			
		\ll
		M(m)^{- r_{\lambda,p,q}} \|f \|_{W^r_{p,w}(\RR)}  + e^{-Km}\|f \|_{L_{q,w}(\RR)}
\\&			
		\ll
		m^{- r_{\lambda,p,q}} \|f \|_{W^r_{p,w}(\RR)}. 
				\end{aligned}
	\end{equation}
	\hfill
	\end{proof}

We  introduce B-spline quasi-interpolation operators for functions on $\RR$. For a given even positive number $2\ell$ denote by  $M_{2\ell}$ the symmetric cardinal B-spline of order $2\ell$ with support $[-\ell,\ell]$ and 
knots at the integer points $-\ell,...,-1, 0, 1,...,\ell$. It is well-known that
\begin{equation} \label{M_2ell=}
M_{2\ell}(x) =  \frac{1}{(2\ell - 1)!} \sum_{k= 0}^{2\ell}(-1)^k \binom{2\ell}{k} (x - k + \ell)_+^{2\ell - 1},
\end{equation}
where $x_+:= \max(0,x)$ for $x \in \RR$ (see, e.g.,  \cite[(4.1.12)]{Chui92}).
Through this paper, we fix the even number $2\ell$ and use the abbreviation $M:= M_{2\ell}$.

Let $\Lambda = \{\lambda(j)\}_{|j| \le j_0}$ be a given finite even sequence, i.e., 
$\lambda(-j) = \lambda(j)$ for some $j_0\ge \ell - 1$. 
We define the linear operator $Q$ for functions $f$ on $\RR$ by  
\begin{equation} \label{def:Q}
	Qf(x):= \ \sum_{s \in \ZZ} \sum_{|j| \le j_0} \lambda (j) f(s-j)M(x-s).
\end{equation} 
The operator $Q$ is local and bounded in $C(\RR)$  (see \cite[p. 100--109]{Chui92}).
An operator $Q$ of the form \eqref{def:Q} is called a 
quasi-interpolation operator  if  it reproduces 
$\Pp_{2\ell-1}$, i.e., $Qf = f$ for every $f \in \Pp_{2\ell-1}$, where $\Pp_m$ denotes the set of  polynomials of degree at most $m$.
Notice that $Qf$ can be written in the form:
\begin{equation} \label{L-representation}
	Qf(x)  \ = \ 
	\sum_{s \in \ZZ} f(s)L(x-s), \ \forall x \in \RR, 
	\end{equation}
where 
\begin{equation} \label{L}
	L(x):= \   
	\sum_{|j| \le j_0} \lambda (j) M(x - j).
\end{equation}

We present some well-known examples of B-spline quasi-interpolation operators. 
A piecewise linear interpolation operator is defined as
\begin{equation} \label{PWlinear-intepolation}
	Qf(x):= \ \sum_{s \in \ZZ} f(s) M(x-s), 
\end{equation} 
where $M$ is the   
symmetric piecewise linear B-spline with support $[-1,1]$ and 
knots at the integer points $-1, 0, 1$ ($\ell =1$).
It is related to the classical Faber-Schauder basis of the hat functions. 
Another example is the cubic quasi-interpolation operator 
\begin{equation}  \label{Q4}
	Qf(x):= \ \sum_{s \in \ZZ} \frac {1}{6} \{- f(s-1) + 8f(s) - f(s+1)\} M(x-s), 
\end{equation} 
where $M$ is the symmetric cubic B-spline with support $[-2,2]$ and 
knots at the integer points $-2, -1, 0, 1, 2$ ($\ell=2$). For more examples of B-spline quasi-interpolation, see \cite{Chui92,BISS2005}.

If $A$ is an operator in the space of functions on $\RR$, 
we define the operator $A_h$ for $h > 0$ by
	\begin{equation} \label{A_h}
A_h
:= \ 
\sigma_h \circ A \circ \sigma_{1/h}
\end{equation}
where $\sigma_hf(x) = \ f(x/h)$.
With this definition,  we have
\begin{equation} \nonumber
	Q_hf(x)  \ = \ 
	\sum_{s \in \ZZ} \sum_{|j| \le j_0} \lambda (j) f(h(s-j))M(h^{-1}x - s), \ \forall x \in \RR.
\end{equation}

Throughout of the present paper, for a fixed number $0 < \rho < 1$, we make use of the notation
	\begin{equation} \label{h_m}
h_m:= \rho a_m/m = \rho \nu_\lambda m^{1/\lambda - 1}, \ \ x_k:= kh_m   \ \  
\forall m \in \NN, \ \forall  k \in \ZZ.
\end{equation}
 We introduce the truncated equidistant,  compact-support $B$-spline quasi-interpolation operator 
 $Q_{\rho,m}$ 
  for $m \in \NN$ by
  \begin{equation}\label{def:Q_{rho,m}}
  Q_{\rho,m}f(x)
  	:=
  	\begin{cases}
  	Q_{h_m}f(x)&  \ \ \text{if} \ \ x \in [-\rho a_m,\rho a_m], \\
  	0 &  \ \ \text{if} \ \ x \notin [-\rho a_m,\rho a_m].
  	\end{cases}	 
  \end{equation}
  By the definition,
 	\begin{equation}\label{Q_{rho,m}:=}
 		Q_{\rho,m}f(x) \ = \ 
 	\sum_{|s| \le m + \ell -1}\sum_{|j| \le j_0} \lambda (j) f(x_{s-j})M(h_m^{-1}x - s) 
 	\ \ \forall x \in [-\rho a_m,\rho a_m], \ \forall m \in \NN.
 \end{equation}
The function $Q_{\rho,m}f$ is constructed  from $2(m + \ell +j_0) - 1$ values of $f$ at the points $x_k$, 
$|k| \le m + \ell + j_0 - 1$, and 
	\begin{equation}\label{suppQ_{rho,m}}
	\supp Q_{\rho,m}f
	=
	[-\rho a_m, \rho a_m].
\end{equation}

The following theorem gives an upper bound for the approximation error by B-spline quasi-interpolation operators $Q_{\rho,m}$.

\begin{theorem} \label{thm:f-Q_{rho,m}f}
	Let $1 \le p,q \le \infty$, $r \le 2\ell$ and $r_{\lambda,p,q} > 0$.	
	Let $\rho$ be any fixed positive number satisfying the condition
\begin{equation}\label{rho<max}
	\rho < \max\brac{1, \,\frac{2 \ell - 1}{\nu_\lambda^\lambda (\ell + j_0)a\lambda}}^{1/\lambda}.
\end{equation}
Then we have that
	\begin{equation}\label{f-Q_{rho,m}f,polynomial}
		\|f - Q_{\rho,m} f\|_{L_{q,w}(\RR)} 
		\ll
		 m^{- r_{\lambda,p,q}} \|f \|_{W^r_{p,w}(\RR)} \ \ \forall f \in W^r_{p,w}(\RR),   \ \forall m \in \NN.  
	\end{equation}
\end{theorem}

\begin{proof}
Fix a positive number $\rho$ satisfying \eqref{rho<max}.	Let 	$ f  \in W^r_{p,w}(\RR)$.  We have  by \eqref{suppQ_{rho,m}}
	\begin{equation}\label{f-Q}
		\|f - Q_{\rho,m} f\|_{L_{q,w}(\RR)} 
		\le  \|f - Q_{\rho,m} f\|_{L_{q,w}([-\rho a_m,\rho a_m])}  
		+ \|f \|_{L_{q,w}(\RR \setminus [-\rho a_m,\rho a_m])}. 
	\end{equation}
	For the second term in the right-hand side, we have by 
	Lemma \ref{lemma:|f|_{L_{q,w}(RR setminus [-rho a_m, rho a_m])}}
		\begin{equation}\nonumber
		\|f \|_{L_{q,w}(\RR \setminus [-\rho a_m,\rho a_m])}
		\ll
		m^{- r_{\lambda,p,q}} \|f \|_{W^r_{p,w}(\RR)}.  
	\end{equation}
	Hence to prove \eqref{f-Q_{rho,m}f,polynomial} it is sufficient to show that for the first term in the right-hand side of \eqref{f-Q}, it holds
	\begin{equation}\label{f-Q_{rho,m}f,T_{rho,m}}
		\|f - Q_{\rho,m} f\|_{L_{q,w}([-\rho a_m,\rho a_m])}
		\ll
		m^{- r_{\lambda,p,q}} \|f \|_{W^r_{p,w}(\RR)}.  
	\end{equation}
 By  \eqref{suppQ_{rho,m}} we have
		\begin{equation}\label{f-Q_{rho,m}f,x_kQ}
		\|f - Q_{\rho,m} f\|_{L_{q,w}([-\rho a_m,\rho a_m])}^q
		=
		\sum_{k=-m}^{m -1} \|f - Q_{\rho,m} f\|_{L_{q,w}([x_k,x_{k+1}])}^q.
	\end{equation}
	Let us estimate each term in the sum of the last equation. For a given $k \in \ZZ$,
let  
		\begin{equation}\label{T_rf}
T_rf(x):= \sum_{s=0}^{r-1}\frac{1}{s!} f^{(s)}(x_k)(x - x_k)^s
	\end{equation}
 be the $r$th  Taylor polynomial of  $f$ at $x_k$. Let a number  
 $k=-m,...,m -1$ be given. We  assume $x_k \ge 0$. The case when $x_k < 0$ can be treated similarly.
Then for every $ x \in  [x_k,x_{k+1}]$, 
$$
f(x) - Q_{\rho,m}f = f (x) - T_rf(x) - Q_{\rho,m}[f (x) - T_rf(x)],
$$
since $Q_{\rho,m}$ reproduces on $[x_k,x_{k+1}]$ polynomials in $\Pp_{2\ell -1}$ and $r \le 2\ell$.
Hence,
		\begin{equation}\label{f-Q_{rho,m}f,x_ka}
 \|f - Q_{\rho,m} f\|_{L_{q,w}([x_k,x_{k+1}])}
 \le
 \|f - T _rf\|_{L_{q,w}([x_k,x_{k+1}])} +
 \|Q_{\rho,m}(f -  T_rf)\|_{L_{q,w}([x_k,x_{k+1}])}. 
 \end{equation}
For the Taylor polynomial $T_rf$ and $ x \in  [x_k,x_{k+1}]$, we have the well-known formula 
(see, e.g., \cite[(5.6), page 37]{DeLo93B})
$$
f(x) - T_rf(x) = \frac{1}{(r-1)!}\int_{x_k}^x f^{(r)}(t) (x-t)^{r-1}\rd t.
$$ 
Hence,
$$
|f(x) - T_rf(x)|w(x) 
\le 
\int_{x_k}^x |f^{(r)}(t)w(t) (x-t)^{r-1}|\rd t.
$$
Applying H\"older's inequality we find for $x \in [x_k, x_{k+1}]$,
		\begin{equation}\label{f-Q_{rho,m}f,x_kb}
|f(x) - T_rf(x)|w(x)  
\le h_m^{r - 1/p} \big\|f^{(r)}\big\|_{L_{p,w}([x_k, x_{k+1}])}.
\end{equation}
Taking the norm of  $L_q([x_k, x_{k+1}])$ of the both sides in this inequality, we receive
		\begin{equation}\label{f-Q_{rho,m}f,x_kc}
	\|f - T_rf\|_{L_{q,w}([x_k,x_{k+1}])} 
\ll
 m^{-r'_{\lambda,p,q}} \big\|f^{(r)}\big\|_{L_{p,w}([x_k, x_{k+1}])} \ \ \forall k \in \ZZ,
\end{equation}
where
	\begin{equation}\label{r'}
r'_{\lambda,p,q}:= (r-1/p+1/q)(1-1/\lambda). 
\end{equation}
Let $g \in C_w(\RR)$. By  \eqref{Q_{rho,m}:=} and \eqref{M_2ell=} for  $x \in [x_k, x_{k+1}]$,
\begin{equation} \nonumber
	Q_{\rho,m}g(x)  \ = \ 
	\sum_{|s - k|\le \ell -1}\sum_{|j| \le j_0}  
	\sum_{i =0}^{2\ell} c_{i,j} h_m^{1-2\ell}g(x_{s - j}) (x - x_{s + i  - \ell})_+^{2\ell - 1}, 
\end{equation}
where 
\begin{equation} \label{c_{i,j}:=}
c_{i,j}:= \frac{1}{(2\ell - 1)!}\lambda(j)(-1)^i \binom{2\ell}{i}.
\end{equation}
 We rewrite the last equality in a more compact form as
\begin{equation} \label{Q_{rho,m}g}
	Q_{\rho,m}g(x)  \ = \ 
	\sum_{(s,i,j) \in J_k^Q} c_{i,j} 	F_{\xi,\eta}g(x) \ \
	\forall x \in [x_k, x_{k+1}], 
\end{equation}
where 
	\begin{equation}\label{J_k^Q} 
J_k^Q:= \brab{(s,i,j): |s - k|\le \ell -1; \  i = 0,1,..., 2\ell; \   |j|\le j_0},
\end{equation}
	\begin{equation}\label{eta,xi}
	\xi := s + i - \ell, \ \ \eta:= s- j,
\end{equation}
and
	\begin{equation}\nonumber
	F_{\xi,\eta}g(x):=	g(x_\eta) h_m^{1-2\ell}(x - x_\xi)_+^{2\ell - 1}.
\end{equation}

With the fixed number $\rho$ satisfying  \eqref{rho<max}, let us  show that 
	\begin{equation}\label{(x - x_s)_+}
h_m^{1 - 2\ell} (x - x_\xi)_+^{2\ell - 1}w(x)
	\ll 
w(x_\eta)  \ \
\forall x \in [x_k, x_{k+1}], \ \ (s,i,j) \in J_k^Q.
\end{equation}
If $\xi\ge k + 1$, as $(x - x_\xi)_+ = 0$ for $x \in [x_k, x_{k+1}]$, this inequality is trivial. If $\xi <  k + 1$ and $ \eta \le k$, then $w(x) \le w(x_\eta)$  and for $(s,i,j) \in J_k^Q$,
$$
(x - x_\xi)_+^{2\ell - 1} 
\le 
(x_{k+1} - x_{k - 3\ell -1})_+^{2\ell - 1} 
\ll  
h_m^{2\ell - 1}
$$
for every $x \in [x_k, x_{k+1}]$.  Hence we obtain \eqref{(x - x_s)_+}.
Consider the remaining case when $\xi  < k + 1\le \eta$.
 For the function 
 $$
 \phi(x):= (x - x_\xi)^{2\ell - 1}w(x),
 $$
 we have 
 
$$
\phi'(x)= 
 (x - x_\xi)^{2\ell - 2}w(x)
\big[(2\ell-1)  -  a \lambda x^{\lambda - 1}(x - x_\xi) \big].
$$
Since for $\lambda >1$, the function $a \lambda x^{\lambda - 1}(x - x_\xi)$ is continuous, strictly increasing on $[x_\xi,\infty )$, and ranges from $0$ to $\infty$ on this interval, there exists a unique point  $t \in (x_\xi,\infty )$ such that $\phi'(t)=0$,  $\phi'(x)>0$ for $x < t $ and $\phi'(x)<0$ for $x >t$. By definition,
$$
\phi'(x_\eta)= 
(x_\eta - x_\xi)^{2\ell - 2}w(x_\eta)
\big[(2\ell-1)  -  a \lambda x_\eta^{\lambda - 1}(x_\eta - x_\xi) \big].
$$
We have 
$$
x_\eta \le x_{k + j_0} \le (k + j_0) h_m \le (m - \ell + j_0)\rho a_m/m
\le \rho a_m,
$$ 
\begin{equation}\label{x_eta - x_xiQ}
x_\eta - x_\xi =  (\eta - \xi )h_m \le (\ell + j_0 )\rho a_m/m,
\end{equation}
 and 
$a_m = (\nu_\lambda m)^{1/\lambda}$. Hence, by using the condition \eqref{rho<max} we derive
$$
a\lambda (x_\eta - x_\xi) x_\eta^{\lambda - 1} 
\le 
(\ell + j_0 )a\lambda (\rho a_m/m)(\rho a_m)^{\lambda - 1}
= 
(\ell + j_0)a\lambda \nu_\lambda^\lambda \rho^\lambda
<
2\ell - 1,
$$
or, equivalently, $\phi'(x_\eta) >0$. This means that $x_\eta \in (x_\xi, t)$ and, therefore,  $\phi'(x)>0$  
for every $x \in [x_\xi, x_\eta]$. It follows that the function $\phi$ is increasing on the interval 
$[x_\xi, x_\eta]$. In particular, we have for every $x \in [x_k,x_{k+1}] \subset [x_\xi, x_\eta]$,
$$
(x - x_\xi) w(x) \le (x_\eta - x_\xi) w(x_\eta), 
$$
which together with \eqref{x_eta - x_xiQ} implies \eqref{(x - x_s)_+}. With $\eta,\xi$ as in \eqref{eta,xi},
we obtain by \eqref{(x - x_s)_+},
	\begin{equation}\nonumber
|F_{\xi,\eta}(f - T_rf)(x)|w(x)
	\le 
|(f - T_rf)(x_\eta)|w(x_\eta) \ \
\forall x \in [x_k, x_{k+1}], \ \ \forall (s,i,j) \in J_k^Q.
\end{equation}
By applying  \eqref{f-Q_{rho,m}f,x_kb} to the right-hand side we get 
	\begin{equation}\label{|F_{xi,eta}(x)(f - T_rf)|w(x)}
|F_{\xi,\eta}(x)(f - T_rf)|w(x)
	\le 
h_m^{r-1/p} \big\|f^{(r)}\big\|_{L_{p,w}([x_{\eta - 1}, x_\eta])}\ \
\forall x \in [x_k, x_{k+1}], \ \ \forall (s,i,j) \in J_k^Q.
\end{equation}
 Hence, similarly to \eqref{f-Q_{rho,m}f,x_kc} we derive
		\begin{equation}\nonumber
	\|F_{\xi,\eta}(f - T_rf)\|_{L_{q,w}([x_k,x_{k+1}])} 
	\ll
	m^{-r'_{\lambda,p,q}}  \big\|f^{(r)}\big\|_{L_{w,p}([x_{\eta-1}, x_\eta])},
\end{equation}
which together with \eqref{Q_{rho,m}g} implies 
\begin{equation} \label{|Q_{rho,m}(f -  T_rf)|}
	\|Q_{\rho,m}(f -  T_rf)\|_{L_{q,w}([x_k,x_{k+1}])} 	
	\ll
	m^{-r'_{\lambda,p,q}}  	\sum_{(s,i,j) \in J_k^Q} \big\|f^{(r)}\big\|_{L_{w,p}([x_{\eta-1}, x_\eta])}.
\end{equation}
From the last inequality, \eqref{f-Q_{rho,m}f,x_ka} and \eqref{f-Q_{rho,m}f,x_kc} it follows that
\begin{equation} \label{|f - Q_{rho,m}f|}
	\|f - Q_{\rho,m}(f )\|_{L_{q,w}([x_k,x_{k+1}])} 	
	\ll
	m^{-r'_{\lambda,p,q}}  	\brac{\big\|f^{(r)}\big\|_{L_{w,p}([x_k, x_{k+1}])} +
		\sum_{(s,i,j) \in J_k^Q} \big\|f^{(r)}\big\|_{L_{w,p}([x_{s-j-1}, x_{s-j}])}}.
\end{equation}
Notice that   $\ell,j_0$  and, therefore, $c_{i,j}$ and $|J_k^Q| \le 2\ell(2\ell - 1)(2j_0 + 1)$ are constants. Hence, taking account the definition of $J_k^Q$ and $-m\le k \le m -1,$
from\eqref{f-Q_{rho,m}f,x_kQ} we derive that
		\begin{equation}\nonumber
	\|f - Q_{\rho,m} f\|_{L_{q,w}([-\rho a_m,\rho a_m])}
	\ll
	m^{-r'_{\lambda,p,q}} 	\brac{\sum_{k=-m - j_0}^{m + j_0 - 1}
		\|f \|_{W^r_{p,w}([x_k,x_{k+1}]}^q}^{1/q} =: A_m.
\end{equation}
For $1\le p \le q \le \infty$, obviously,
		\begin{equation}\label{f-Q_{rho,m}f,a_m(a)}
A_m
	\le 
 m^{-r'_{\lambda,p,q}} 	
 \brac{\sum_{k=-m - j_0}^{m + j_0 - 1}
 	\|f \|_{W^2_{p,w}([x_\xi,x_\eta])}^p}^{1/p} 
	\le  m^{- r_{\lambda,p,q}}\|f \|_{W^r_{p,w}(\RR)}.	
\end{equation}
For $1\le  q < p \le \infty$, by Young's inequality,
\begin{equation}\label{f-Q_{rho,m}f,a_m(b)}
	A_m
	\ll 
 m^{-r'_{\lambda,p,q}} 	m^{1/q-1/p}
 \brac{\sum_{k=-m - j_0}^{m + j_0 - 1}
 	\|f \|_{W^2_{p,w}([x_\xi,x_\eta])}^p}^{1/p} 
	\le  m^{- r_{\lambda,p,q}}\|f \|_{W^r_{p,w}(\RR)}.	
\end{equation}
From the last three inequalities 
\eqref{f-Q_{rho,m}f,T_{rho,m}} is implied. The theorem has been proven. 
	\hfill	
\end{proof}	

\begin{remark} \label{remark1}		
It is worth emphasizing the following.
	In Theorem~\ref{thm:f-Q_{rho,m}f},  since the parameters $\lambda$, $a$, $\nu_\lambda$, $\ell$ and $j_0$ are already specified, a  value of  $\rho >0$ satisfying the condition \eqref{rho<max} can be chosen explicitly. Moreover, because the B-splines $M(h_m^{-1}x - s)$ employed in the definition \eqref{Q_{rho,m}:=} of the B-spline quasi-interpolation operators 
	$ Q_{\rho,m}$ are explicitly constructed, these operators are also  determined  constructively.		
	This remark  also holds  for  the B-spline interpolation operators 
	$ P_{\rho,m}$ in Theorem~\ref{thm:f-P_{rho,m}f}, the associated quadratures $I^Q_{\rho,m}$ and $I^P_{\rho,m}$ in Theorem~\ref{thm:I_n}, 
B-spline inequalities in  
	Theorems~\ref{thm:MarcinkiewiczInequality}--\ref{thm:BernsteinInequality} and  multidimensional  generalizations of these interpolations and quadratures in Theorems \ref{thm:f-Q_{rho,m}f,d} and \ref{thm:I_n,d}.
\end{remark}

\subsection{B-spline interpolation}
\label{B-spline interpolation}

We have seen in the previous section that the B-spline quasi-interpolation algorithms $Q_{\rho,m}$ possess good local and approximation properties for functions in the Sobolev space $W^r_{p,w}(\RR)$. 
However, they do not have   interpolation property, except in the case of piece-wise linear interpolation when $Q$ is defined as in \eqref{PWlinear-intepolation}. In this subsection, we construct equidistant,  compact-support B-spline  algorithms having the same properties as $Q_{\rho,m}$, which interpolate functions at the points $x_k$, $|k| \le m$.

We present a construction of B-spline interpolation with compact-support and local properties suggested in \cite[pp. 114--117]{Chui92}. For a given  integer $\ell >1$ we define $\kappa:= \lceil \log_2 2\ell - 1\rceil$ and 
 the operator $R$ for functions $f \in C_w(\RR)$ by 
\begin{equation} \label{def:R}
	Rf(x):= \ M(0)^{-1} \sum_{s \in \ZZ} f(s)  M(2^\kappa (x-s)).
\end{equation} 
For example, if $\ell = 2$, then 
\begin{equation} \label{R4}
	Rf(x) = \ \frac{3}{2} \sum_{s \in \ZZ} f(s)  M_4(2(x-s)).
\end{equation} 
The operator $R$ is local and bounded in $C_w(\RR)$. Moreover, it interpolates  $f$ at integer points $s  \in \ZZ$, i.e., $Rf(s) = f(s)$. 
However, $R$ does not reproduce polynomials in
$\Pp_{2\ell-1}$, and hence does not have a good approximation property.

We define the blended operator $P$ by:
\begin{equation} \nonumber
	P:= \ R + Q - RQ,
\end{equation} 
where recall, $Q$ is the B-spline quasi-interpolation operator defined as in \eqref{def:Q}.

By the definitions we get for $f \in C_w(\RR)$,
\begin{equation} \label{RQ}
	\begin{aligned}
		RQ f(x) 
		=  \sum_{s \in \ZZ} \sum_{|j| \le j_0} 
		\sum_{|i - s|\le \ell} M(0)^{-1}\lambda (j)M(i - s) f(s-j)  M(2^\kappa (x-s) - i). 
	\end{aligned}
\end{equation} 

From \eqref{def:Q}, \eqref{def:R} and \eqref{RQ}, we obtain the explicit formula for $P$
\begin{equation} \nonumber
	\begin{aligned}
		Pf(x)
		& = 
		  \sum_{s \in \ZZ} M(0)^{-1}f(s)  M(2^\kappa (x-s))
\\& \ \ \ 
        + \sum_{s \in \ZZ} \sum_{|j| \le j_0} \lambda (j) f(s-j) M(x-s)
\\& \ \ \ 		        
		-   \sum_{s \in \ZZ} \sum_{|j| \le j_0} 
		\sum_{|i - s|\le \ell} M(0)^{-1} \lambda (j)M(i - s) f(s-j)  M(2^\kappa (x-s) - i). 
	\end{aligned}
\end{equation} 
The operator $P$ is local and bounded in $C_w(\RR)$  (see \cite[p. 100--109]{Chui92}).
It reproduces 
$\Pp_{2\ell-1}$, i.e., $Pf = f$ for every $f \in \Pp_{2\ell-1}$.  Moreover, $Pf$  interpolates $f$  at the  integer points $s  \in \ZZ$. For $h>0$, the scaled operator $P_hf$  interpolates $f$ at the  points $sh$ for  $s  \in \ZZ$, i.e.,  $P_hf(sh) = f(sh)$ for  $s  \in \ZZ$.

For example, for $\ell = 2$ and  $P$  based on the cubic B-spline quasi-interpolation operator $Q$ given by \eqref{Q4} and the interpolation operator $R$ given by \eqref{R4}, we can present  $P$ as
\begin{equation} \nonumber
	Pf(x) = \ \sum_{s \in \ZZ} \sum_{|j| \le 4} \lambda_{s - j}f(j) M_4(2x- s), 
\end{equation} 
where
$
\lambda_0 := 29/72, \ \lambda_{\pm 1}:= 7/12, \ \lambda_{\pm 2}:= -1/8, \ \lambda_{\pm 3}:=-1/12, \ \lambda_{\pm 4}:=1/48.
$
 
In the next step, we use the construction of B-spline interpolation for weighted sampling recovery of functions   
$f  \in W^r_{p,w}(\RR)$. In the same manner as the definition of  $Q_{\rho,m}$ in \eqref{def:Q_{rho,m}}, we define the  truncated equidistant compact-support $B$-spline interpolation operator 
$P_{\rho,m}$  for $\forall m \in \NN$:
   \begin{equation*}
 	P_{\rho,m}f(x)
 	:=
 	\begin{cases}
 		P_{h_m}f(x)&  \ \ \text{if} \ \ x \in [-\rho a_m,\rho a_m], \\
 		0 &  \ \ \text{if} \ \ x \notin [-\rho a_m,\rho a_m],
 	\end{cases}	 
 \end{equation*}
 where recall, $h_m$ is as in \eqref{h_m}.
 By the definition, we have for every 
$m \in \NN$  and $x \in [-\rho a_m,\rho a_m]$,
 \begin{equation}\label{P_{rho,m}:=}
	\begin{aligned}
P_{\rho,m}f(x)&: = \ R_{\rho,m}f + Q_{\rho,m}f - (RQ)_{\rho,m}f
	\\ & = 
	\sum_{|s| \le m + \ell -1}M(0)^{-1}f(x_s)  M(2^\kappa h_m^{-1}x- 2^\kappa s)
	\\& \ \ \ \ \
	+ \sum_{|s| \le m + \ell -1} \sum_{|j| \le j_0} \lambda (j) f(x_{s-j}) M(h_m^{-1}x - s)
	\\& \ \ \ \ \		        
	-   \sum_{|s| \le m + \ell -1}\sum_{|j| \le j_0} 
	\sum_{|i - s|\le \ell} M(0)^{-1} \lambda (j)M(i - s) f(x_{s-j}) M(2^\kappa h_m^{-1}x- 2^\kappa s - i).
	\end{aligned}
 \end{equation}
The function $P_{\rho,m}f$ is constructed  from $2(m + \ell + j_0) - 1$ values of $f$  at the points $x_k$, $|k| \le m  + \ell + j_0 - 1$,
 	\begin{equation}\label{suppP_{rho,m}}
 	\supp P_{\rho,m}f
 	=
 	[-\rho a_m, \rho a_m].
 \end{equation}
 	$P_{\rho,m}f(x)= P_{h_m}f(x)$ for $x \in [x_{- m},x_{m}]$, and hence, $P_{\rho,m}f$ interpolates $f$ at the $2m +1$ points $x_k$ for $|k| \le m$, i.e., 
$$
P_{\rho,m}f(x_k) = f(x_k), \ \ |k| \le m.
$$

The following theorem gives an upper bound for the approximation error by B-spline interpolation operators $P_{\rho,m}$.

\begin{theorem} \label{thm:f-P_{rho,m}f}
	Let $1 \le p,q \le \infty$, $r \le 2\ell$ and $r_{\lambda,p,q} > 0$.	
	Let $\rho$ be any fixed positive number satisfying the condition
	\begin{equation}\label{rho<maxP}
		\rho < \max\brac{1, \,\frac{2 \ell - 1}
			{\nu_\lambda^\lambda (2^\kappa j_0 +  2\ell) 2^{-\kappa \lambda}a\lambda}}^{1/\lambda}.
	\end{equation}
	Then one can determine explicitly a number 	$\rho:= \rho(a,\lambda,\ell,j_0)$  with $0 < \rho < 1$, so that
	\begin{equation}\label{f-P_{rho,m}f,polynomial}
		\|f - P_{\rho,m} f\|_{L_{q,w}(\RR)} 
		\ll
		m^{- r_{\lambda,p,q}} \|f \|_{W^r_{p,w}(\RR)} \ \ \forall f \in W^r_{p,w}(\RR),   \ \forall m \in \NN.  
	\end{equation}
\end{theorem}

 The technique of the proof of this theorem is similar to that of the proof  Theorem~\ref{thm:f-Q_{rho,m}f}, but more complicate. It is given in Appendix \ref{Proof of Theorem ref{thm:f-P_{rho,m}f}}.

\begin{remark}
{\rm	
	To construct the truncated B-spline  interpolation operator $P_{\rho,m}$, it is necessary to learn the sampled values of $f$ at the $2(m + \ell +  j_0) -1$ points $x_k$ for $|k| \le m + \ell + j_0 - 1$, while $P_{\rho,m}f$ interpolates $f$ at only  the  $2m  + 1$ points $x_k$ for $|k| \le m$. Thus, these interpolation points are strictly less than the required sampled function  values,  except the single case of the  piece-wise linear interpolation when $\ell=1$ and $j_0 =0$ (cf. \eqref{PWlinear-intepolation}). For $\ell \ge 2$, this divergence can be overcome by the following modification of $P_{\rho,m}$ which reduces the sample points.

If $f$ is a continuous function on $\RR,$ let  
$f^-$ and $f^+$
be the $(2\ell - 1)$th Lagrange polynomials interpolating $f$
at the $2\ell$ points $x_{-m},..., x_{-m + 2\ell -1},$ and 
at the $2\ell$ points
$x_{m - 2\ell + 1},..., x_m,$  respectively. Put 
\begin{equation*} 
	\bar{f}(x):= \
	\begin{cases}
		f^-(x), \ & x  \in (-\infty,\rho a_m), \\
		f(x), \ & x  \in [-\rho a_m,\rho a_m] \\
		f^+(x), \ & x  \in (\rho a_m, + \infty).
	\end{cases}
\end{equation*} 	
We define the  truncated equidistant $B$-spline interpolation operator 
$
\bar{P}_{\rho,m}
$ 
for $m \ge 2\ell$ by
\begin{equation}\nonumber
	\bar{P}_{\rho,m}f: = \ P_{\rho,m}\bar{f}
\end{equation}
	In the same manner, we define the operator $\bar{Q}_{\rho,m}$. By the construction, the functions $\bar{P}_{\rho,m}f$ and $\bar{Q}_{\rho,m}f$ are constructed  from the  values of $f$  at the   $2m+ 1$ points $x_k$, $|k| \le m$,
$$
\supp \bar{P}_{\rho,m}f =\supp \bar{Q}_{\rho,m}f =[-\rho a_m,\rho a_m],
$$
and $\bar{P}_{\rho,m}f$ interpolates $f$ at the same $2m + 1$ points $x_k$ for $|k| \le m$, i.e., 
$$
\bar{P}_{\rho,m}f(x_k) = f(x_k), \ \ |k| \le m.
$$
Moreover, if $1 \le p,q \le \infty$, $r \le 2\ell$ and $r_{\lambda,p,q} > 0$, then in a way similar to the proof of Theorem~\ref{thm:f-P_{rho,m}f}, we can prove that there exists $0 < \rho < 1$ such that
	\begin{equation}\nonumber
		\|f - \bar{S}_{\rho,m} f\|_{L_{q,w}(\RR)} 
		\ll
		m^{- r_{\lambda,p,q}} \|f \|_{W^r_{p,w}(\RR)} \ \ \forall f \in W^r_{p,w}(\RR),   \ \forall m \ge 2\ell,
	\end{equation}
where $\bar{S}_{\rho,m}$	denotes either $\bar{P}_{\rho,m}$ or $\bar{Q}_{\rho,m}$.
}
\end{remark}

\section{Optimality of sampling algorithms}
\label{Optimality of sampling algorithms}

\subsection{Weighted B-spline inequalities}
\label{Weighted B-spline inequalities}
In this subsection, we prove some weighted  Marcinkiewicz-,  Nikol'skii- and Bernstein-type inequalities for scaled cardinal B-splines, which are interesting themselves and which will be used for establishing the optimality of the B-spline quasi-interpolation operator $Q_{\rho,m}$ and interpolation operator $P_{\rho,m}$ in the next subsection.

Denote by $S_{\rho,m}$, $m > \ell$,  the subspace in $C_w(\RR)$ of all B-spline $\varphi$ on $\RR$ of the form
\begin{equation}\nonumber
	\varphi(x) = \ 
\sum_{|s| \le m - \ell} b_s M_{\rho,m,s}(x), \ \forall x \in \RR, 
\end{equation}
where  $M_{\rho,m,s}(x):= 	M(h_m^{-1}x -  s)$ and recall, $h_m$ is as in \eqref{h_m}. 
In what follows, to emphasize the dependence of the coefficients  $b_s$ on $\varphi$, we will write $b_s:=b_s(\varphi)$. 
Since  $Q_{\rho,m}$ reproduce on the interval $[-\rho a_m, \rho a_m]$ polynomials from $\Pp_{2\ell - 1}$,  we can see that $Q_{\rho,m}\varphi(x) = \varphi(x)$ and, therefore, 
\begin{equation}\label{varphi=(1)}
	\varphi(x) = \ 
	\sum_{|s| \le m - \ell}\sum_{|j| \le j_0} \lambda (j) \varphi(x_{s-j})M(h_m^{-1}x - s)\  \ 
	\forall \varphi \in S_{\rho,m}, \ \forall x \in \RR.
\end{equation}
Moreover, the B-splines $(M_{\rho,m,s})_{|s| \le m-\ell}$ form a basis in $S_{\rho,m}$,  $\dim S_{\rho,m}= 2(m - \ell) + 1$ and
\begin{equation}\label{supp varphi}
	\supp \varphi
	=
	[-\rho a_m,\rho a_m]  \  \ \forall \varphi \in S_{\rho,m}.
\end{equation}
For $1 \le p \le \infty$, $n \in \NN_0$ and a sequence $(c_s)_{|s| \le n}$ we introduce the weighted norm
\begin{equation}\nonumber
	\|(c_s)\|_{p,w,n}
	:=
\brac{\sum_{|s| \le n} |w(x_s)c_s|^p}^{1/p}
\end{equation}
for $1 \le p  < \infty$ with the corresponding modification when $p = \infty$.

\begin{theorem} \label{thm:MarcinkiewiczInequality} 
	Let $1 \le p\le \infty$. Let $\rho$ be any fixed positive number satisfying the condition
\eqref{rho<max}.
Then there hold the Marcinkiewicz-type inequalities
	\begin{equation}\label{MarcinkiewiczInequality} 
		\|\varphi\|_{L_{p,w}(\RR)} 
		\asymp
			m^{(1/\lambda - 1)/p}\|(\varphi(x_s))\|_{p,w,m} 
	\asymp
	m^{(1/\lambda - 1)/p}\|(b_s(\varphi))\|_{p,w, m-\ell}
	 \ \ 	\forall \varphi \in S_{\rho,m},   \ \forall m  \ge \ell. 
	\end{equation}
\end{theorem}

The proof of this theorem is given in Appendix \ref{Proof of Theorem ref{thm:MarcinkiewiczInequality}}.

	\begin{theorem} \label{thm:NikolskiiInequality} 
		Let $1 \le p,q\le \infty$.	Let $\rho$ be any fixed positive number satisfying the condition
			\eqref{rho<max}.
			Then   there holds the  Nikol'skii-type inequality
		\begin{equation*}
			\|\varphi\|_{L_{q,w}(\RR)} 
			\ll
			m^{\delta_{\lambda,p,q}}	\|\varphi\|_{L_{p,w}(\RR)}  \ \ \  
			\forall \varphi \in S_{\rho,m},   \ \forall m  \ge \ell. 
		\end{equation*}
	\end{theorem}
\begin{proof}
	This theorem is a consequence of Theorem \ref{thm:MarcinkiewiczInequality}. Let us prove it for completeness.
	Indeed, let 	
	$ \varphi \in S_{\rho,m}$ and $m  \ge \ell$.  We have by Theorem \ref{thm:MarcinkiewiczInequality} 
	for  $1 \le p \le q\le \infty$,
	\begin{equation}\nonumber
		\begin{aligned}
		\|\varphi\|_{L_{q,w}(\RR)} 
		&\asymp
		m^{(1/\lambda - 1)/q}\|(\varphi(x_s))\|_{q,w,m} 
		\ll
			m^{(1/\lambda - 1)/q}\|(\varphi(x_s))\|_{p,w,m}
			\\&
				\asymp	 
		m^{(1/\lambda - 1)/q}m^{(1/\lambda - 1)/p}	\|\varphi\|_{L_{p,w}(\RR)}
		= 
		m^{\delta_{\lambda,p,q}}	\|\varphi\|_{L_{p,w}(\RR)}, 				
		\end{aligned}
	\end{equation}
and for  $1 \le q < p\le \infty$,
\begin{equation}\nonumber
	\begin{aligned}
		\|\varphi\|_{L_{q,w}(\RR)} 
		&\asymp
		m^{(1/\lambda - 1)/q}\|(\varphi(x_s))\|_{q,w,m} 
		\\&	\le
		m^{(1/\lambda - 1)/q} (2(m - \ell) +1)^{1/q - 1/p}\|(\varphi(x_s))\|_{p,w,m}
		\\&
		\asymp	 
		m^{(1/\lambda - 1)/q}m^{1/q - 1/p}m^{(1/\lambda - 1)/p}	\|\varphi\|_{L_{p,w}(\RR)}
		= 
		m^{\delta_{\lambda,p,q}}	\|\varphi\|_{L_{p,w}(\RR)}. 				
	\end{aligned}
\end{equation}
		\hfill	
\end{proof}	
\begin{theorem} \label{thm:BernsteinInequality}
	Let $1 \le p\le \infty$, $r \le 2\ell$ and $r_\lambda > 0$.	Let $\rho$ be any fixed positive number satisfying the condition
		\eqref{rho<max}.
		Then there holds the Bernstein-type inequality
	\begin{equation}\label{varphi^{(r)}}
		\|\varphi^{(r)}\|_{L_{p,w}(\RR)} 
		\ll
		m^{r_\lambda} \|\varphi\|_{L_{p,w}(\RR)} \ \ \  
		\forall \varphi \in S_{\rho,m},   \ \forall m  \ge \ell. 
	\end{equation}
\end{theorem}

The proof of this theorem is given in Appendix \ref{Proof of Theorem ref{thm:BernsteinInequality}}.

\subsection{Optimality}
\label{Optimality}

In this subsection, we prove the optimality of the constructed B-spline  quasi-interpolation and interpolation algorithms in terms of the sampling $n$-widths $\varrho_n(\bW^r_{p,w}(\RR),  L_{q,w}(\RR))$, and compute the exact convergence rate of these sampling $n$-widths. 

\begin{theorem} \label{thm:S_n}
	Let  $1\le p,q \le \infty$ and $r_{\lambda,p,q} >0$.   For any $n \in \NN$, let $m(n)$ be the largest integer such that $2(m +\ell + j_0)  - 1 \le n$.  Let the sampling algorithm $S_n \in \Ss_n$ be either the B-spline quasi-interpolation operator $Q_{\rho,m(n)}$ or the B-spline interpolation operator $P_{\rho,m(n)}$. Then  $S_n$ is asymptotically optimal for the sampling $n$-widths $\varrho_n\big(\bW^r_{p,w}(\RR),  L_{q,w}(\RR)\big)$ and
		\begin{equation}\label{rho_n-d=1}
			\varrho_n\big(\bW^r_{p,w}(\RR),  L_{q,w}(\RR)\big) 
			\asymp 
				\sup_{f\in \bW^r_{p,w}(\RR)} \big\|f - S_nf\big\|_{ L_{q,w}(\RR)} 
			\asymp
			n^{- r_{\lambda,p,q}}.
		\end{equation}
\end{theorem}

The exact convergence rate of  $\varrho_n\big(\bW^r_{p,w}(\RR),  L_{q,w}(\RR)\big)$ as in \eqref{rho_n-d=1} of  Theorem \ref{thm:S_n} has been proven in \cite{DD2024} for $1 <p< \infty$ and $1\le q \le \infty$. 
	This exact convergence rate is achieved by  generalized methods  of  truncated Lagrange interpolation  from \cite{MN2010} which is completely different from the methods proposed in the present paper.  Moreover, the lower bound in 	\eqref{rho_n-d=1} has been proven in \cite{DD2024} for the cases $1 \le p < q \le \infty$ and $1< p < \infty$, $ p \ge q$ which still do not cover all the cases in this theorem.	Let us prove Theorem \ref{thm:S_n}.

\begin{proof}
The upper bound in \eqref{rho_n-d=1} follows from Theorems \ref{thm:f-Q_{rho,m}f} and \ref{thm:f-P_{rho,m}f}.

Let us prove the lower bound  in 	\eqref{rho_n-d=1}  by a method  distinct  from that in \cite{DD2024},  employing the weighted B-spline inequalities in Section \ref{Weighted B-spline inequalities}.
From the definition \eqref{rho_n} we have the following inequality which is often used  for lower estimation of sampling $n$-widths. 
If $F$ is a set of continuous functions on $\RR$ and $X$ is a normed space of functions on $\RR$,	then we have
\begin{equation} \label{rho_n>}
	\varrho_n(F, X) \ge \inf_{\brab{x_1,...,x_n} \subset \RR} \ \sup_{f\in F: \ f(x_i)= 0,\ i =1,...,n} 
	\|f \|_X.
\end{equation}
  
We first consider the case $1 \le q \le p \le \infty$.
For a given $n \in \NN$, we take a number $m > \ell$ satisfying  the inequality $2m + 1 > 4\ell(n + 1)$. Let $\brab{\xi_1,...,\xi_n} \subset \RR$ be arbitrary $n$ points.  Then there are numbers $s_1, ..., s_n \in \ZZ$ such that $|2\ell s_j| \le m - \ell$ and 
$$
\brab{\xi_1,...,\xi_n} \cap \brac{\cup_{j=1}^n [x_{2\ell s_j}, x_{2\ell(s_j + 1)}]} = \varnothing.
$$
Consider the B-spline
	\begin{equation}\label{varphi(x) :=}
\varphi(x) := Cn^{-r_\lambda  - 1/(p\lambda)} \sum_{j =1}^n M(h_m^{-1}x - x_{2\ell s_j}).
\end{equation}
By the construction $\varphi(\xi_i) = 0$, $i =1,...,n$. 
By Theorem \ref{thm:BernsteinInequality}
 there  a number  $0 < \rho < 1$  such that 
	\begin{equation}\nonumber
	\|\varphi^{(r)}\|_{L_{p,w}(\RR)} 
	\le
	C' m^{r_\lambda} \|\varphi\|_{L_{p,w}(\RR)} \ \ \  
	 \forall m  \ge \ell. 
\end{equation}
Again, by the construction and the relation $m \asymp n$,
\begin{equation}\label{|varphi|_{L_{p,w}(RR)}^p}
	\begin{aligned}
		\|\varphi\|_{L_{p,w}(\RR)}^p 
		&=
		C^pn^{-pr_\lambda - 1/\lambda}\sum_{j =1}^n \int_{ x_{2\ell s_j}-\ell}^{ x_{2\ell s_j} +\ell}M(h_m^{-1} x - x_{2\ell s_j})^p \rd x
		\\ &
		= C^pn^{-pr_\lambda  - 1/\lambda} h_m \sum_{j =1}^n \int_{-\ell}^{\ell} M(x)^p \rd x
		\le C^p (2 \ell n) n^{-pr_\lambda  - 1/\lambda} \rho (\nu_\lambda m)^{1/\lambda}/m
		\\ &
		= C^p K m^{-pr_\lambda },
	\end{aligned}
\end{equation}
where $K$ is a constant depending on $\ell, \lambda,\rho$ only. This means that one can choose a constant $C$ independent of $m$ and $n$, in the definition \eqref{varphi(x) :=} of $\varphi$ so that 
$\varphi \in \bW^r_{p,w}(\RR)$. 
By using the inequality  \eqref{rho_n>} in a similar way as in \eqref{|varphi|_{L_{p,w}(RR)}^p} we obtain
\begin{equation}\label{|varphi|_{L_{q,w}(RR)}^q}
	\begin{aligned}
	\varrho_n\big(\bW^r_{p,w}(\RR), L_{q,w}(\RR)\big)^q 	
&\ge 
	\|\varphi\|_{L_{q,w}(\RR)}^q 
	\\	&=
		C^q n^{-qr_\lambda - q/(p\lambda)}
		\sum_{j =1}^n \int_{ x_{2\ell s_j}-\ell}^{ x_{2\ell s_j} +\ell}M(h_m^{-1} x - x_{2\ell s_j})^q \rd x
		\\ &
		= C^q n^{-qr_\lambda - q/(p\lambda)} h_m \sum_{j =1}^n \int_{-\ell}^{\ell} M(x)^q \rd x
		\\ &
			\gg C^q (2 \ell n)  n^{-qr_\lambda - q/(p\lambda)} \rho (\nu_\lambda m)^{1/\lambda}/m
		\\ &
		\gg n^{-qr_\lambda - q/(p\lambda) + 1 + 1/\lambda -1}
	\\ &
	= n^{-q\brac{r_\lambda -   (1/\lambda)(1/q - 1/p)}} =  n^{-qr_{\lambda,p,q}}.
	\end{aligned}
\end{equation}

We now prove the lower bound in \eqref{rho_n-d=1} for the case $1 \le p \le q \le \infty$. 
For a given $n \in \NN$, we take a number $m > \ell$ satisfying  the inequality $2m  + 1 > 2\ell(n + 1)$. Let $\brab{\xi_1,...,\xi_n} \subset \RR$ be arbitrary $n$ points.  Then there is a number $s_0 \in \ZZ$ such that $|2\ell s_0| \le m - \ell$ and 
$$
\brab{\xi_1,...,\xi_n} \cap [x_{2\ell s_0}, x_{2\ell(s_0 + 1)}] = \varnothing.
$$
Consider the B-spline
$$
\psi(x) := Cn^{-r_\lambda  + (1 - 1/\lambda)/p}M(h_m^{-1}x - x_{2\ell s_0}).
$$
By the construction $\varphi(\xi_i) = 0$, $i =1,...,n$. 
By Theorem \ref{thm:BernsteinInequality}
there  exists a number  $0 < \rho < 1$  such that 
\begin{equation}\nonumber
	\|\varphi^{(r)}\|_{L_{p,w}(\RR)} 
	\le
	C' m^{r_\lambda} \|\varphi\|_{L_{p,w}(\RR)} \ \ \  
	\forall m  \ge \ell. 
\end{equation}
Again, by the construction and the relation $m \asymp n$,
\begin{equation}\label{|varphi|_{L_{p,w}(RR)}^p, p<q}
	\begin{aligned}
		\|\varphi\|_{L_{p,w}(\RR)}^p 
		&=
		C^p n^{-pr_\lambda  + (1 - 1/\lambda)}
		\int_{ x_{2\ell s_0}-\ell}^{ x_{2\ell s_0} +\ell}M(h_m^{-1} x - x_{2\ell s_j})^p \rd x
		\\ &
		= C^p n^{-pr_\lambda  + (1 - 1/\lambda)} h_m \int_{-\ell}^{\ell} M(x)^p \rd x
		\le C^p (2 \ell) n^{-pr_\lambda +  1 - 1/\lambda} \rho (\nu_\lambda m)^{1/\lambda}/m
		\\ &
		= C^p K m^{-pr_\lambda },
	\end{aligned}
\end{equation}
where $K$ is a constant depending on $\ell, \lambda,\rho$ only. This means that 
one can choose a constant $C$ independent of $m$ and $n$, in the definition \eqref{varphi(x) :=}
 so that 
$\varphi \in \bW^r_{p,w}(\RR)$. 
By using the inequality  \eqref{rho_n>} in the same way as \eqref{|varphi|_{L_{p,w}(RR)}^p, p<q} we obtain
\begin{equation}\label{|varphi|_{L_{q,w}(RR)}^q}
	\begin{aligned}
		\varrho_n\big(\bW^r_{p,w}(\RR), L_{q,w}(\RR)\big)^q 	
		&\ge 
		\|\varphi\|_{L_{q,w}(\RR)}^q 
		\\	&=
		C^q 
		n^{- qr_\lambda  + q(1 - 1/\lambda)/p} 
		\int_{ x_{2\ell s_0}-\ell}^{ x_{2\ell s_0} +\ell}M(h_m^{-1} x - x_{2\ell s_j})^q \rd x
		\\ &
		= C^q n^{- qr_\lambda  + q(1 - 1/\lambda)/p} h_m  \int_{-\ell}^{\ell} M(x)^q \rd x
		\\ &
		\gg C^q    n^{- qr_\lambda  + q(1 - 1/\lambda)/p} \rho (\nu_\lambda m)^{1/\lambda}/m
		\\ &
		\gg  n^{- qr_\lambda  + q(1 - 1/\lambda)/p + 1/\lambda  - 1}
		\\ &
		= n^{-q\brac{r_\lambda  -   (1/\lambda)(1/p - 1/q)}} =  n^{-qr_{\lambda,p,q}}.
	\end{aligned}
\end{equation}
	\hfill
\end{proof}

\begin{remark} \label{remark3.5}
	It is interesting to  study the computational cost for constructing  the  equidistant,  compact-supported  B-spline quasi-interpolation and interpolation sampling algorithms $Q_{\rho,m}$ and  $P_{\rho,m}$ in the sense of  \cite[Sectionn 4.1.2 Algorithms and Their Cost]{NoWo08}. However, this topic lies outside the scope of the present paper.
	
	Theorem~\ref{thm:S_n} can be interpreted in terms of the computational complexity in the following sense. For $\varepsilon >0$, we define the quantity $n_\varepsilon$ of computational complexity for approximate linear sampling recovery of $f \in \bW^r_{p,w}(\RR)$ with accuracy~$\varepsilon$ by
	\begin{equation}\nonumber
		n_\varepsilon
		:= \ 
		\inf \brab{n\in \NN: \, \exists S_n \in \Ss_n: 
			\sup_{f\in \bW^r_{p,w}(\RR)} \big\|f - S_n f\big\|_{ L_{q,w}(\RR)} \le \varepsilon }.
	\end{equation}
It is evident that $n_\varepsilon$ represents a necessary number of  samples  of 
$f \in \bW^r_{p,w}(\RR)$ to construct a linear sampling 
 algorithm that approximates $f$ with accuracy~$\varepsilon$ in the norm of $L_{q,w}(\RR)$. 	 Under the assumptions and notations of Theorem~\ref{thm:S_n},  we derive that
\begin{equation}\nonumber
	n_\varepsilon
	\asymp \varepsilon^{-1/r_{\lambda,p,q}}, \ \ 0 < \varepsilon \le \varepsilon_0,
\end{equation}
for  some $\varepsilon_0 >0$. Moreover,
if the sampling algorithm $S_{n_\varepsilon} \in \Ss_{n_\varepsilon}$ is either the B-spline quasi-interpolation operator $Q_{\rho,m({n_\varepsilon})}$ or the B-spline interpolation operator $P_{\rho,m({n_\varepsilon})}$, then
	\begin{equation}\nonumber
		\sup_{f\in \bW^r_{p,w}(\RR)} \big\|f - S_{n_\varepsilon} f\big\|_{ L_{q,w}(\RR)} \le \varepsilon .
\end{equation}
\end{remark}

\section{Numerical integration}
\label{Numerical integration}
In this section, we prove that  the equidistant quadratures   generated from  the truncated B-spline  quasi-interpolation and interpolation algorithms, are asymptotically optimal in terms of
$\Int_n\big(\bW^r_{p,w}(\RR)\big)$, and compute the exact convergence rate of 
$\Int_n\big(\bW^r_{p,w}(\RR)\big)$.

The sampling operators $Q_{\rho,m}$ and  $P_{\rho,m}$ generate in a natural way the weighted quadrature operators $I^Q_{\rho,m}$ and  $I^P_{\rho,m}$ by the formula \eqref{I_kf-generated}:
 	\begin{equation}\nonumber
	I^Q_{\rho,m}f: = \ 
	\int_{\RR}Q_{\rho,m}f(x) w(x) \rd x; \ \ I^P_{\rho,m}f: = \ 
	\int_{\RR}P_{\rho,m}f(x) w(x) \rd x.
\end{equation}
Indeed, from the definitions, we can see that $I^Q_{\rho,m}f$ and  $I^P_{\rho,m}f$ with 
$2(m +\ell + j_0) - 1 \le n$ are quadratures of the form \eqref{I_kf} from $\Ii_n$. In particular, by \eqref{L-representation}
	\begin{equation}\nonumber
	I^Q_{\rho,m}f 
		  \ = \ 
		\sum_{|s| \le m +\ell +j_0 -1} \lambda_s f(x_s),
\end{equation}
where 
$$ \lambda_s := \int_{\RR}L_s(x)w(x) \rd x, \ \ 
L_s(x):= L(h_m^{-1}x-s) \chi_{[-\rho a_m, \rho a_m]}(x),
$$
 $\chi_{[-\rho a_m, \rho a_m]}$ is the characteristic function of $[-\rho a_m, \rho a_m]$ and $L$ is as in \eqref{L}.
\begin{theorem} \label{thm:I_n}
	Let  $1\le p\le \infty$, $r \le 2\ell$  and $r_\lambda - (1/\lambda)(1 - 1/p) >0$.   For any $n \in \NN$, let $m(n)$ be the largest integer such that $2(m + \ell +j_0)  - 1 \le n$.  Let the quadrature $I_n \in \Ii_n$ be either   $I^Q_{\rho,m(n)}$ or  $I^P_{\rho,m(n)}$. Then  $I_n$ is asymptotically optimal for  $\Int_n\big(\bW^r_{p,w}(\RR)\big)$ and
	\begin{equation}\label{I_n-thm}
		\Int_n\big(\bW^r_{p,w}(\RR)\big) 
		\asymp 
		\sup_{f\in \bW^r_{p,w}(\RR)} \left|\int_{\RR} f(x) w(x) \, \rd x - I_nf \right|
		\asymp
		n^{- r_\lambda + (1/\lambda)(1 - 1/p)} \ \  \forall n \in \NN.
	\end{equation}
\end{theorem}

	In the  work \cite{DD2023}, we have proven  the exact convergence rate of  $\Int_n(\bW^r_{1,w}(\RR))$ as in \eqref{I_n-thm} of  Theorem~\ref{thm:I_n} for $p=1$. 
	This convergence rate is achieved by  generalized methods  of  truncated  Gaussian quadratures from  \cite{DM2003}. The asymptotically optimal quadrature algorithms proposed in the present paper, are completely different from those in the above cited papers. Let us prove Theorem~\ref{thm:I_n}.

\begin{proof} 
	Let $S_n \in \Ss_n$ be either   $Q_{\rho,m(n)}$ or  $P_{\rho,m(n)}$ which generates  $I^Q_{\rho,m(n)}$ or  $I^P_{\rho,m(n)}$, respectively. We have by Theorem \ref{thm:f-Q_{rho,m}f} or Theorem \ref{thm:f-P_{rho,m}f}
for $q=1$,	
\begin{equation}\nonumber
\sup_{f\in \bW^r_{p,w}(\RR)} \left|\int_{\RR} f(x) w(x) \, \rd x - I_nf \right|
\le
		\sup_{f\in \bW^r_{p,w}(\RR)} \big\|f - S_nf\big\|_{L_{1,w}(\RR)}
		\asymp
		n^{- r_\lambda + (1/\lambda)(1 - 1/p)} \ \ \forall n \in \NN.
	\end{equation}
	This proves the upper bound in \eqref{I_n}.

 In order to prove the lower bound in \eqref{I_n} we  need the following inequality which follows directly from the definition. 
 For  a set  $F$ of continuous functions on $\RR$,	 
 we have
 \begin{equation} \label{Int_n>}
 	\Int_n(F) \ge \inf_{\brab{x_1,...,x_n} \subset \RR} \ \sup_{f\in F: \ f(x_i)= 0,\ i =1,...,n}\bigg|\int_{\RR} f(x) w(x) \, \rd x\bigg|.
 \end{equation}
 Let $\brab{\xi_1,...,\xi_n} \subset \RR$ be arbitrary $n$ points. Consider the B-spline $\varphi$ defined as in \eqref{varphi(x) :=}. As shown in the proof of Theorem \ref{thm:S_n}
  $\varphi(\xi_i) = 0$, $i =1,...,n$, and
 there  exist a number  $0 < \rho < 1$ and a constant $C$ independent of $m$ and $n$, in the definition \eqref{varphi(x) :=} so that 
 $\varphi \in \bW^r_{p,w}(\RR)$. By the construction, \eqref{Int_n>} and \eqref{|varphi|_{L_{q,w}(RR)}^q},
\begin{equation} \nonumber
	\Int_n(F) \ge \bigg|\int_{\RR} \varphi(x) w(x) \, \rd x\bigg|
	=
	\|\varphi\|_{L_{1,w}(\RR)} \gg n^{-r_{\lambda,p,1}} 
	= 	
	n^{- r_\lambda + (1/\lambda)(1 - 1/p)}.
\end{equation}
	\hfill
\end{proof}	

	\begin{remark} \label{remark4.2}				
		The construction of the quadratures $I^Q_{\rho,m}$ or  $I^P_{\rho,m}$ depends on several factors, in particular, the smoothness $r$, integrability parameter $p$ of function $f$, and the used B-splines in the  quasi-interpolation and interpolation operators 
		$Q_{\rho,m}$ or  $P_{\rho,m}$, respectively. In practice, the
		smoothness $r$ and integrability parameter $p$ of function $f$ are frequently unknown, and one often  only has access
		to  its certain samples at a finite set of nodes.  In Theorem~\ref{thm:I_n}, these smoothness and integrability parameter are assumed to be known, and the degree $2\ell$ of used B-splines can be selected as the minimal integer satisfying $r \le 2 \ell$.	
		By contrast, truncated Gaussian quadratures -- constructed from subsets of the zeros of orthonormal polynomials with respect to Freud-type measures \cite{DD2023,MN2010} -- and the  truncated trapezoidal rule -- based on equidistant nodes \cite{GKS2024,KSG2023} -- are independent of these parameters. Consequently, they are well suited for numerical weighted integration of a function even when its exact regularity is unknown.
		This property is  an advantage of the  truncated Gaussian quadrature and truncated trapezoidal  rule over the quadratures  $I^Q_{\rho,m}$ and  $I^P_{\rho,m}$ generated from B-spline quasi-interpolation and interpolation.
		In particular, the former approaches offer robustness to uncertainty in the regularity of $f$,
		 whereas the latter depend on the (often unknown) smoothness  of the integrand.			
	\end{remark}

\section{Multidimensional generalization}
\label{High dimensional generalization}

In this section, we formulate a generalization of  the results in the previous sections
to multidimensional case when $d > 1$, which can be proven in a similar way with certain modifications.

Let $Q$ be an one-dimensional B-spline quasi-interpolation operator defined  as in  \eqref{def:Q}.
We define the linear operator $Q_d$ for functions $f$ on $\RR^d$ by  
\begin{equation} \label{def:Q_d}
	Q_df(\bx):=  \sum_{\bs \in \ZZ^d} \sum_{|\bj| \le\bj_0} \lambda (\bj) f(\bs-\bj)M(\bx-\bs), 
	\ \forall \bx \in \RR^d,
\end{equation} 
where $\bj_0:= (j_0,...,j_0)$ and
$
	M(\bx):= \prod_{i=1}^d M(x_i), \ \ \lambda(\bj):= \prod_{i=1}^d \lambda(j_i)
$
and $|\bj|:= (|j_1|,...,|j_d|)$ for $\bj \in \ZZd$.
The operator $Q_d$ can be seen as the  product $\prod_{i=1}^d Q_i$, where $Q_i=Q$ is the one-dimensional operator applied to $f$ as a univariate function in $x_i$ while the other variables fixed.
The operator $Q_d$ is local and bounded in $C(\RR^d)$.
An operator $Q$ of the form \eqref{def:Q_d} is called a 
quasi-interpolation operator in $C(\RR^d)$  if  it reproduces 
$\Pp_{2\ell-1}^d$, i.e., $Q_df = f$ for every $f \in \Pp_{2\ell-1}^d$, where $\Pp_m^d$ denotes the set of $d$-variate polynomials of degree at most $m$ in each variable. Clearly, if $Q$ is an one-dimensional B-spline quasi-interpolation operator, then $Q_d$ is a $d$-dimensional B-spline quasi-interpolation operator.

If $A$ is an operator in the space of functions on $\RR^d$, 
the operator $A_h$ for $h > 0$ is defined in the same manner as in \eqref{A_h} for the one-dimensional case.
With this definition,  we have
\begin{equation} \nonumber
	Q_{d,h}f(\bx)  \ = \ 
	\sum_{\bs \in \ZZ^d} \sum_{|\bj| \le \bj_0} \lambda (\bj) f(h(\bs-\bj))M(h^{-1}\bx - \bs), \ \forall \bx \in \RR^d.
\end{equation}
For $\forall m \in \NN$ and $0 < \rho < 1$, we make use of the notation
\begin{equation}\nonumber
 \bx_\bk:= h_m \bk :=(h_mk_1,...,h_mk_d),   
\ \bk \in \ZZ^d,
\end{equation}
where recall, $h_m:= \rho a_m/m$.
We introduce the $d$-dimensional truncated equidistant $B$-spline quasi-interpolation operator 
$Q_{d,\rho,m}$
for $m \in \NN$ by
  \begin{equation}\nonumber
	Q_{d,\rho,m}f(\bx)
	:=
	\begin{cases}
		Q_{d,h_m}f(\bx)&  \ \ \text{if} \ \ \bx \in [-\rho a_m,\rho a_m]^d, \\
		0 &  \ \ \text{if} \ \ \bx \notin [-\rho a_m,\rho a_m]^d.
	\end{cases}	 
\end{equation}
By the definition,

\begin{equation}\nonumber
	Q_{d,\rho,m}f(\bx): = \ 
	\sum_{|\bs| \le \bm +\bell - \bone}\sum_{|\bj| \le \bj_0} \lambda (\bj) f(\bx_{\bs-\bj})M(h_m^{-1}\bx - \bs), \ \forall \bx \in \in [-\rho a_m,\rho a_m]^d, \ \forall m \in \NN,
\end{equation}
where $\bone:= (1,...,1)$ and $\bell:= (\ell,...,\ell)$.
The function $Q_{d,\rho,m}f$ is constructed  from $[2(m + \ell + j_0) - 1]^d$ values $\bx_\bk$, $|\bk| \le \bm +\bell +\bj_0 - \bone$, and 
\begin{equation}\nonumber
	\supp Q_{d,\rho,m}f
	=
	[-\rho a_m,\rho a_m]^d.
\end{equation}
The $d$-dimensional truncated equidistant $B$-spline interpolation operator 
$P_{\rho,d,m}$ is defined in the same manner.
It possesses the same properties as  $Q_{d,\rho,m}$ and, moreover,  $P_{\rho,d,m}f$ interpolates $f$  at the points $\bx_\bk$ for $|\bk| \le \bm$, i.e., 
$$
P_{\rho,d,m}f(\bx_\bk) = f(\bx_\bk), \ \ |\bk| \le \bm.
$$

We make use of  the notations:  $\bW^{r,\operatorname{iso}}_{p,w}(\RRd)$ denotes the unit ball in 
$W^{r,\operatorname{iso}}_{p,w}(\RRd)$;
$$
r_{\lambda,d}:= r_\lambda/d; \ \
	r_{\lambda,p,q,d} 
	:=
	r_{\lambda,d} - \delta_{\lambda,p,q}. 
$$
\begin{theorem} \label{thm:f-Q_{rho,m}f,d}
	Let $1 \le p,q \le \infty$, $r \le 2\ell$ and $r_{\lambda,p,q,d} > 0$.	Let $S_{\rho,d,m}$ be either $Q_{d,\rho,m}$ or $P_{\rho,d,m}$. Let $\rho$ be any fixed number satisfying \eqref{rho<max} for  $Q_{d,\rho,m}$, or  \eqref{rho<maxP} for $P_{\rho,d,m}$, respectively.
	Then one can determine explicitly a number 	$\rho:= \rho(a,\lambda,\ell,j_0,d)$  with $0 < \rho < 1$, so that 
	\begin{equation}\nonumber
		\|f - S_{\rho,d,m} f\|_{L_{q,w}(\RRd)} 
		\ll
		m^{- dr_{\lambda,p,q,d}} \|f \|_{W^{r,\operatorname{iso}}_{p,w}(\RRd)} \ \ \forall f \in W^{r,\operatorname{iso}}_{p,w}(\RRd),   \ \forall m \in \NN.  
	\end{equation}
\end{theorem}

\begin{theorem} \label{thm:S_n,d}
	Let  $1\le p,q \le \infty$, $r \le 2\ell$ and $r_{\lambda,p,q,d} >0$.   For any $n \in \NN$, let $m(n)$ be the largest integer such that $[2(m + \ell + j_0)  - 1]^d \le n$.  Let the sampling operator $S_n \in \Ss_n$ be either the B-spline quasi-interpolation operator $Q_{d,\rho,m(n)}$ or the B-spline interpolation operator $P_{d,\rho,m(n)}$. Then  $S_n$ is asymptotically optimal for the sampling $n$-widths 
	$\varrho_n\big(\bW^{r,\operatorname{iso}}_{p,w}(\RRd),  L_{q,w}(\RRd)\big)$ and
	\begin{equation}\nonumber
		\varrho_n\big(\bW^{r,\operatorname{iso}}_{p,w}(\RRd),  L_{q,w}(\RRd)\big) 
		\asymp 
		\sup_{f\in \bW^{r,\operatorname{iso}}_{p,w}(\RRd)} \big\|f - S_nf\big\|_{ L_{q,w}(\RRd)} 
		\asymp
		n^{- r_{\lambda,p,q,d}}.
	\end{equation}
\end{theorem}

The sampling operators $Q_{d,\rho,m}$ and  $P_{d\rho,m}$ generate the weighted quadrature operators $I^Q_{d,\rho,m}$ and  $I^P_{d,\rho,m}$ by the formula \eqref{I_kf-generated} as
\begin{equation}\nonumber
	I^Q_{d,\rho,m}f: = \ 
	\int_{\RRd}Q_{d,\rho,m}f(\bx) w(\bx) \rd \bx; \ \ I^P_{d,\rho,m}f: = \ 
	\int_{\RRd}P_{d,\rho,m}f(\bx) w(\bx) \rd \bx, 
\end{equation}
respectively.

\begin{theorem} \label{thm:I_n,d}
	Let  $1\le p\le \infty$, $r \le 2\ell$ and $r_{\lambda,d} - (1/\lambda)(1 - 1/p) >0$.   For any $n \in \NN$, let $m(n)$ be the largest integer such that $[2(m +\ell +j_0)  - 1]^d \le n$.  Let the quadrature operator $I_n \in \Ii_n$ be either   $I^Q_{d,\rho,m(n)}$ or  $I^P_{d,\rho,m(n)}$. Then  $I_n$ is asymptotically optimal for  $\Int_n\big(\bW^{r,\operatorname{iso}}_{p,w}(\RRd)\big)$ and
	\begin{equation}\nonumber
		\Int_n\big(\bW^{r,\operatorname{iso}}_{p,w}(\RRd)\big) 
		\asymp 
		\sup_{f\in \bW^{r,\operatorname{iso}}_{p,w}(\RRd)} \left|\int_{\RRd} f(\bx) w(\bx) \, \rd \bx - I_nf \right|
		\asymp
		n^{- r_{\lambda,d} + (1/\lambda)(1 - 1/p)} \ \  \forall n \in \NN.
	\end{equation}
\end{theorem}

Denote by $S_{d,\rho,m}$, $m > \ell$,  the subspace in $C_w(\RRd)$ of all B-spline $\varphi$ on $\RRd$ of the form
\begin{equation}\nonumber
	\varphi(\bx) = \ 
	\sum_{|\bs| \le \bm - \bell} b_\bs(\varphi) M_{d,\rho,m,s}(\bx), \ \forall \bx \in \RRd, 
\end{equation}
where  $M_{d,\rho,m,\bs}(\bx):= 	M(h_m^{-1}\bx -  \bs)$. Similarly  to the univariate case, we have
\begin{equation}\nonumber
	\varphi(x) = \ 
	\sum_{|\bs| \le \bm - \bell}\sum_{|\bj| \le \bj_0} \lambda (\bj) \varphi(\bx_{\bs - \bj})M(h_m^{-1}\bx - \bs)\  \ 
	\forall \varphi \in S_{d,\rho,m}, \ \forall \bx \in \RRd.
\end{equation}
Moreover, the B-splines $(M_{d,\rho,m,\bs})_{|\bs| \le \bm - \bell}$ is a basis in $S_{d,\rho,m}$,  $\dim S_{d,\rho,m}= [2(m - \ell) + 1]^d$ and
\begin{equation}\nonumber
	\supp \varphi
	=
	[-\rho a_m,\rho a_m]^d  \  \ \forall \varphi \in S_{d,\rho,m}.
\end{equation}
For $1 \le p \le \infty$, $n \in  \NN$ and a sequence $(c_\bs)_{|\bs| \le \bn}$  we introduce the norm
\begin{equation}\nonumber
	\|(c_\bs)\|_{p,d,w,n}
	:=
	\brac{\sum_{|\bs| \le \bn} |w(\bx_\bs)c_\bs|^p}^{1/p}
\end{equation}
for $1 \le p  < \infty$ with the corresponding modification when $p = \infty$, where $\bn:= (n,...,n)$.

We have also the following multidimensional Marcinkiewicz- Nikol'skii- and Bernstein-type inequalities. Let $1 \le p,q\le \infty$.	Let $\rho$ be any fixed number satisfying \eqref{rho<max}.	Then    for every 
$m \ge \ell$ and every $\varphi \in S_{\rho,d,m}$
	\begin{equation}\nonumber
		\|\varphi\|_{L_{p,w}(\RR)} 
		\asymp
		m^{d(1/\lambda - 1)/p}\|(\varphi(\bx_\bs))\|_{p,d,w,m}
		\asymp
		m^{d(1/\lambda - 1)/p}\|(b_\bs(\varphi))\|_{p,d,w, d,m-\ell};
	\end{equation}
	\begin{equation}	\nonumber
		\|\varphi\|_{L_{q,w}(\RRd)} 
		\ll
		m^{d\delta_{\lambda,p,q}}	\|\varphi\|_{L_{p,w}(\RRd)};  
	\end{equation}
	\begin{equation}\nonumber
		\|\varphi\|_{W^{r,\operatorname{iso}}_{p,w}(\RRd)} 
		\ll
		m^{r_\lambda} \|\varphi\|_{L_{p,w}(\RRd)}.
	\end{equation}
	
	\begin{remark}
		{\rm	
			All the results in Sections \ref{Optimality of sampling algorithms}--\ref{High dimensional generalization} are still hold true if  the truncated B-spline  quasi-interpolation and interpolation operators $Q_{\rho,m}$ and $P_{\rho,m}$ are  replaced by $\bar{Q}_{\rho,m}$ and $\bar{P}_{\rho,m}$, respectively.
		}
	\end{remark}

\bigskip
\noindent
{\bf Acknowledgments:}  
This work is funded by the Vietnam National Foundation for Science and Technology Development (NAFOSTED) in the frame of the NAFOSTED--SNSF Joint Research Project under  Grant 
IZVSZ2$_{ - }$229568.  
 A part of this work was done when  the author was working at the Vietnam Institute for Advanced Study in Mathematics (VIASM). He would like to thank  the VIASM  for providing a fruitful research environment and working condition. The author  gratefully acknowledges the referees for valuable comments and remarks that significantly improved the presentation of this paper.
 
 \appendix
 \section{Appendix} 
 \label{appendix}

\subsection{Proof of Theorem \ref{thm:f-P_{rho,m}f}} 
\label{Proof of Theorem ref{thm:f-P_{rho,m}f}}

\begin{proof}
Fix a positive number $\rho$ satisfying \eqref{rho<maxP}.	By the same argument as in the proof of Theorem \ref{thm:f-Q_{rho,m}f}, to prove \eqref{f-P_{rho,m}f,polynomial} it is sufficient to show that there exists $0 < \rho < 1$ such that 
	\begin{equation}\label{f-P_{rho,m}f,T_{rho,m}}
		\|f - P_{\rho,m} f\|_{L_{q,w}([-\rho a_m,\rho a_m])}
		\ll
		m^{- r_{\lambda,p,q}} \|f \|_{W^r_{p,w}(\RR)} \ \ \forall f \in W^r_{p,w}(\RR),   \ \forall m \in \NN.   
	\end{equation}
	We have by \eqref{suppP_{rho,m}},
	\begin{equation}\nonumber
		\|f - P_{\rho,m} f\|_{L_{q,w}([-\rho a_m,\rho a_m])}^q
		=
		\sum_{k=-m}^{m - 1} \|f - P_{\rho,m} f\|_{L_{q,w}([x_k,x_{k+1}))}^q.
	\end{equation}
	Let us estimate each term in the sum of the last equation. For a given $k \in \ZZ$,
	let $T_rf$
	be the $r$th  Taylor polynomial of  $f$ at $x_k$ given as in \eqref{T_rf}. Let a number  $k=-m,...,m-1$ be given. We assume $x_k \ge 0$. The case when $x_k < 0$ can be treated similarly.
	For every $ x \in  [x_k,x_{k+1}]$, 
	$$
	f(x) - P_{\rho,m}f(x) = f (x) - T_rf(x) - P_{\rho,m}[f (x) - T_rf(x)],
	$$
	since the operator $P_{\rho,m}$ reproduces on $[x_k,x_{k+1}]$ the polynomial $T_rf$.
	Hence,
	\begin{equation}\label{f-P_{rho,m}f,x_ka}
		\|f - P_{\rho,m} f\|_{L_{q,w}([x_k,x_{k+1}])}
		\le
		\|f - T _rf\|_{L_{q,w}([x_k,x_{k+1}])} +
		\|P_{\rho,m}(f -  T_rf)\|_{L_{q,w}([x_k,x_{k+1}])}. 
	\end{equation}
	The first term in the right-hand side can be estimated as in \eqref{f-Q_{rho,m}f,x_kc}.	For the second term we have 
	%
	for every $ x \in  [x_k,x_{k+1}]$, 
	\begin{equation}\label{P(f-T)}
		\begin{aligned}
			\|P_{\rho,m}\brac{f - T_rf}\|_{L_{q,w}([x_k,x_{k+1}])} 
			& \le 
			\|R_{\rho,m}\brac{f - T_rf}\|_{L_{q,w}([x_k,x_{k+1}])} \\
			& \ \ \ +
			\|Q_{\rho,m}\brac{f - T_rf}\|_{L_{q,w}([x_k,x_{k+1}])} \\
			& \ \ \ +  
			\|R_{\rho,m}Q_{\rho,m}\brac{f - T_rf}\|_{L_{q,w}([x_k,x_{k+1}])}.
		\end{aligned}	
	\end{equation}	
	Let us estimate each term in the sum in the right-hand side of the last inequality. The second term can be estimated as in \eqref{|Q_{rho,m}(f -  T_rf)|}.  We estimate the first term.
	Let $g \in C_w(\RR)$. 
	By  \eqref{M_2ell=}  for  $x \in [x_k, x_{k+1}]$,
	\begin{equation} \nonumber
		R_{\rho,m}g(x)  \ = \ 
		\sum_{|s| \le k - 1}  
		\sum_{i =0}^{2\ell} c_i h_m^{1-2\ell}g(x_{ s }) (2^\kappa x - x_{2^\kappa s + i - \ell})_+^{2\ell - 1}, 
	\end{equation}
	where 
	$$
	c_i= \frac{1}{(2\ell - 1)!M(0)}(-1)^i \binom{2\ell}{i}.
	$$
	We rewrite the last equality in a more compact form as
	\begin{equation} \nonumber
		R_{\rho,m}g(x)  \ = \ 
		\sum_{(s,i) \in J_k^R} c_i	\Phi_{\xi,s}(x) \ \
		\forall x \in [x_k, x_{k+1}], 
	\end{equation}
	where  $$J_k^R:= \brab{(s,i): \, |s| \le k - 1; \  i = 0,1,..., 2\ell},$$
	\begin{equation}\label{eta,xiR}
		\xi := 2^\kappa s +  i - \ell, 
	\end{equation}
	and
	\begin{equation}\label{F_{xi,eta}}
		\Phi_{\xi,s}g(x):=	g(x_s) h_m^{1-2\ell}(2^\kappa x - x_\xi)_+^{2\ell - 1}.
	\end{equation}
	Then we have 
	\begin{equation} \label{|R_{rho,m}(f -  T_rf)|}
		\|R_{\rho,m}g\|_{L_{q,w}( [x_k, x_{k+1}])} 	
		\ \ll \ 
		\sum_{(s,i) \in J_k^R} \|\Phi_{\xi,s}g \|_{L_{q,w}([x_k, x_{k+1}])}. 	
	\end{equation}
	By a computation we deduce 
	\begin{equation} \label{|G_{xi,s}|}
		\|\Phi_{\xi,s} g\|_{L_{q,w}( [x_k, x_{k+1}])}
		\ = \
		2^{\kappa/q}\|F_{\xi,s}g \|_{L_{q,w_\kappa}([x_{2^\kappa k},  x_{2^\kappa(k+1)}])},
	\end{equation}
	where $w_\kappa(x):= e^{-a_\kappa|x|^\lambda}$, $a_\kappa:=a2^{-\kappa \lambda}$ and
	\begin{equation}\label{F_{xi,eta}}
		F_{\xi,s}g(x):=	g(x_s) h_m^{1-2\ell}(x - x_\xi)_+^{2\ell - 1}.
	\end{equation}
	Let us prove that 
	\begin{equation}\label{(x - x_s)_+R}
		h_m^{1 - 2\ell} (x - x_\xi)_+^{2\ell - 1}w_\kappa(x)
		\ll 
		w_\kappa(x_s)  \ \
		\forall x \in [x_{2^\kappa k},  x_{2^\kappa(k+1)}], \ \ (s,i) \in J_k^R.
	\end{equation}
	If $\xi\ge 2^\kappa(k+1)$, as $(x - x_\xi)_+ = 0$ for $x \in  [x_{2^\kappa k},  x_{2^\kappa(k+1)}]$, this inequality is trivial. If $\xi< 2^\kappa(k+1)$ and $ s \le 2^\kappa k$, then $w_\kappa(x) \le w_\kappa(x_s)$  and for $(s,i) \in J_k^R$,
	$$
	(x - x_\xi)_+^{2\ell - 1} 
	\le
	(x_{2^\kappa(k+1)} - x_{2^\kappa(k - 2\ell -1) - \ell})_+^{2\ell - 1} 
	\ll  
	h_m^{2\ell - 1}
	$$
	for every $x \in  [x_{2^\kappa k},  x_{2^\kappa(k+1)}]$.  Hence we obtain \eqref{(x - x_s)_+R}.
	Consider the remaining case when $\xi  < 2^\kappa(k+1)\le s$.
	For the function 
	$$
	\phi(x):= (x - x_\xi)^{2\ell - 1}w_\kappa(x),
	$$	
	we have 
	$$
	\phi'(x)= 
	(x - x_\xi)^{2\ell - 2}w_\kappa(x)
	\big[(2\ell-1)  -  a_\kappa \lambda x^{\lambda - 1}(x - x_\xi) \big].
	$$
	Since the function $a_\kappa \lambda x^{\lambda - 1}(x - x_\xi)$ is continuous, strictly increasing on $[x_\xi,\infty )$, and ranges from $0$ to $\infty$ on this interval, there exists a unique point  $t \in (x_\xi,\infty )$ such that $\phi'(t)=0$,  $\phi'(x)>0$ for $x < t $ and $\phi'(x)<0$ for $x >t$. By definition,
	$$
	\phi'(x_s)= 
	(x_s - x_\xi)^{2\ell - 2}w_\kappa(x_s)
	\big[(2\ell-1)  -  a_\kappa \lambda x^{\lambda - 1}(x_s - x_\xi) \big].
	$$
	We have 
	$$
	x_s \le  m  h_m = \rho a_m,
	$$ 
	\begin{equation}\label{x_eta - x_xi}
		x_s - x_\xi = (s - 2^\kappa s - i + \ell)h_m \le \ell h_m,
		=
		\ell \rho a_m/m
	\end{equation}
	and 
	$a_m := \nu_\lambda m^{1/\lambda}$. 
	Hence, by using the condition \eqref{rho<maxP} and $a_\kappa:=a2^{-\kappa \lambda}$ we derive
		$$
		a_\kappa\lambda (x_s - x_\xi) x_\eta^{\lambda - 1} 
		\le 
		a_\kappa\lambda\ell (\rho a_m/m)(\rho a_m)^{\lambda - 1}
		= 
		a2^{-\kappa \lambda}\lambda \nu_\lambda^\lambda \rho^\lambda
		<
		2\ell - 1,
		$$
		or, equivalently, $\phi'(x_\eta) >0$.
	 This means that $x_\eta \in (x_\xi, t)$ and, therefore,  $\phi'(x)>0$  
	for every $x \in [x_\xi, x_\eta]$. It follows that the function $\phi$ is increasing on the interval 
	$[x_\xi, x_\eta]$. 
	In particular, we have for every $x \in  [x_{2^\kappa k},  x_{2^\kappa(k+1)}] \subset [x_\xi, x_s]$,
	$$
	(x - x_\xi) w_\kappa(x) \le (x_\eta - x_\xi) w_\kappa(x_\eta), 
	$$
	which together with \eqref{x_eta - x_xi} implies \eqref{(x - x_s)_+R}. With $\xi,s$ as in \eqref{eta,xiR}
	by \eqref{(x - x_s)_+},
	\begin{equation}\nonumber
		|F_{\xi,s}g(x)|w_\kappa(x)
		\le 
		|g(x_s)|w_\kappa(x_s) \ \
		\forall x \in  [x_{2^\kappa k},  x_{2^\kappa(k+1)}], \ \ \forall (s,i) \in J_k^R.
	\end{equation}
	Applying this inequality for $g = f - T_rf$, in a way similar to \eqref{|F_{xi,eta}(x)(f - T_rf)|w(x)}, we get 
	\begin{equation} \nonumber
		|F_{\xi,s}(f - T_rf)(x)|w_\kappa(x)
		\le 
		h_m^{r-1/p} \big\|f^{(r)}\big\|_{L_{p,w_\kappa}([x_{s-1}, x_s])}\ \
		\forall x \in  [x_{2^\kappa k},  x_{2^\kappa(k+1)}], \ \ \forall (s,i) \in J_k^R.
	\end{equation}
	Hence, analogously to \eqref{f-Q_{rho,m}f,x_kc} we derive
	\begin{equation} \nonumber
		\|F_{\xi,s}(f - T_rf)\|_{L_{q,w_\kappa}( [x_{2^\kappa k},  x_{2^\kappa(k+1)}])} 
		\ll
		m^{-r'_{\lambda,p,q}}  \big\|f^{(r)}\big\|_{L_{p,w_\kappa}([x_{s-1},  x_s])},
	\end{equation}
	where $r'_{\lambda,p,q}$ is as in \eqref{r'}.
	From the last inequality and
	\begin{equation} \nonumber
		\big\|f^{(r)}\big\|_{L_{p,w_\kappa}([x_{s-1}, x_s])}  
		= 	
		2^{r\kappa - 1/p}\big\|f^{(r)}\big\|_{L_{p,w}([2^{-\kappa}x_{s-1}, 2^{-\kappa}x_s])}
	\end{equation}	
	it follows that
	\begin{equation} \nonumber
		\|F_{\xi,\eta}(f - T_rf)\|_{L_{q,w_\kappa}( [x_{2^\kappa k},  x_{2^\kappa(k+1)}])} 
		\ll
		m^{-r'_{\lambda,p,q}}  
		\big\|f^{(r)}\big\|_{L_{p,w}([2^{-\kappa}x_{s-1}, 2^{-\kappa}x_s])},
	\end{equation}
	which together with \eqref{|R_{rho,m}(f -  T_rf)|}--\eqref{|G_{xi,s}|} implies 
	\begin{equation} \label{|R_{rho,m}(f -  T_rf)|(2)}
		\|R_{\rho,m}(f -  T_rf)\|_{L_{q,w}([x_k, x_{k+1}])} 	
		\ll
		m^{-r'_{\lambda,p,q}}  
		\sum_{(s,i) \in J_k^R}	\big\|f^{(r)}\big\|_{L_{p,w}([2^{-\kappa}x_{s-1}, 2^{-\kappa}x_s])}.
	\end{equation}
	
	We now process the estimation of the third term in the right-hand side of \eqref{P(f-T)}. 
	By using  formula \eqref{RQ}, 
	we can rewrite 
	\begin{equation} \label{RQ_{rho,m}(f -  T_rf)}
		(RQ)_{\rho,m}(f -  T_rf)(x)  \ = \ 
		\sum_{(s,i,j) \in J_k^{RQ}} c_{s,i,j} 	G_{\xi,\eta}(f -  T_rf)(x) \ \
		\forall x \in [x_k, x_{k+1}], 
	\end{equation}
	where 
	$$
	c_{s,i,j}:= M(0)^{-1}\lambda (j)M(i).
	$$
	$$
	J_k^{RQ}:= \brab{(s,i,j): \, |s-i| \le k - 1; \  |i | \le \ell, \ |j| \le j_0},
	$$
	\begin{equation}\nonumber
		\xi := 2^\kappa(s-i) + i - \ell, \ \eta:= s- i - j,
	\end{equation}
	and
	\begin{equation} \nonumber
		G_{\xi,\eta}(f -  T_rf)(x):=	(f -  T_rf)(x_\eta) h_m^{1-2\ell}(2^\kappa x - x_\xi)_+^{2\ell - 1}.
	\end{equation}
	
	Let us prove  that 
	\begin{equation}\label{(x - x_s)_+RQ}
		h_m^{1 - 2\ell} (2^\kappa x - x_\xi)_+^{2\ell - 1}w_\kappa(x)
		\ll 
		w_\kappa(x_\eta)  \ \
		\forall x \in [x_k,  x_{k+1}], \ \ (s,i,j) \in J_k^{RQ}.
	\end{equation}
	If $\xi\ge 2^\kappa(k+1)$, as $(2^\kappa x - x_\xi)_+ = 0$ for 
	$x \in [x_k,  x_{k+1}]$, this inequality is trivial. If $\xi< 2^\kappa(k+1)$ and $ \eta \le  k$, then $w_\kappa(x) \le w_\kappa(x_\eta)$  and for $(s,i,j) \in J_k^{RQ}$,
	$$
	(2^\kappa x - x_\xi)_+^{2\ell - 1} 
	\le
	(x_{2^\kappa(k+1)} - x_{2^\kappa(k - 3\ell -1)})_+^{2\ell - 1} 
	\ll  
	h_m^{2\ell - 1}
	$$
	for every $x \in  [x_k,  x_{k+1}]$.  Hence we obtain \eqref{(x - x_s)_+R}.
	Consider the remaining case when $\xi  < 2^\kappa(k+1)$ and 
	$k+1\le \eta$.
	For the function 
	$$
	\phi(x):= (2^\kappa x - x_\xi)^{2\ell - 1}w_\kappa(x),
	$$	
	we have 
	$$
	\phi'(x)= 
	(2^\kappa x - x_\xi)^{2\ell - 2}w_\kappa(x)
	\big[2^\kappa(2\ell-1)  -  a_\kappa \lambda x^{\lambda - 1}(2^\kappa x - x_\xi) \big].
	$$
	Since the function $a_\kappa \lambda x^{\lambda - 1}(x - x_\xi)$ is continuous, strictly increasing on $[x_\xi,\infty )$, and ranges from $0$ to $\infty$ on this interval, there exists a unique point  $t \in (x_\xi,\infty )$ such that $\phi'(t)=0$,  $\phi'(x)>0$ for $x < t $ and $\phi'(x)<0$ for $x >t$. By definition,
	$$
	\phi'(x_\eta)= 
	(2^\kappa x_\eta - x_\xi)^{2\ell - 2}w_\kappa(x_\eta)
	\big[2^\kappa(2\ell-1)  -  a_\kappa \lambda x_\eta^{\lambda - 1}(2^\kappa x_\eta - x_\xi) \big].
	$$
	We have 
	$$
	x_\eta \le  m  h_m = \rho a_m,
	$$ 
	\begin{equation}\label{x_eta - x_xiRQ}
	2^\kappa x_\eta - x_\xi 
	\le 
	(2^\kappa j_0 +  2\ell) h_m,
		=
		(2^\kappa j_0 +  2\ell) \rho a_m/m
	\end{equation}
	and 
	$a_m := \nu_\lambda m^{1/\lambda}$. Hence, by condition \eqref{rho<maxP} and $a_\kappa:=a2^{-\kappa \lambda}$,
	$$
	a_\kappa\lambda (2^\kappa x_\eta - x_\xi) x_\eta^{\lambda - 1} 
	\le 
	a_\kappa\lambda(2^\kappa j_0 +  2\ell) (\rho a_m/m)(\rho a_m)^{\lambda - 1}
	< 2 \ell - 1.
	$$
	This means that $2^\kappa x_\eta \in (x_\xi, t)$ and, therefore,  $\phi'(x)>0$  
		for every $x \in [x_\xi, 2^\kappa x_\eta]$. It follows that the function $\phi$ is increasing on the interval 
		$[x_\xi, 2^\kappa x_\eta]$. In particular, we have for every 
		$x \in [x_k,x_{k+1}] \subset [x_\xi, 2^\kappa x_\eta]$,
		$$
		(2^\kappa x - x_\xi) w(x) \le (2^\kappa x_\eta - x_\xi) w(x_\eta), 
		$$
		which together with \eqref{x_eta - x_xiRQ} implies \eqref{(x - x_s)_+RQ}. 
	
	By using formula 	\eqref{RQ_{rho,m}(f -  T_rf)} , in a way similar to the proof of 
	\eqref{|R_{rho,m}(f -  T_rf)|(2)}, we can establish the bound
	\begin{equation} \label{|RQ_{rho,m}(f -  T_rf)|}
		\|(RQ)_{\rho,m}(f -  T_rf)\|_{L_{q,w}([x_k, x_{k+1}])} 	
		\ll
		m^{-r'_{\lambda,p,q}}  
		\sum_{(s,i,j) \in J_k^{RQ}}	\big\|f^{(r)}\big\|_{L_{p,w}([2^{-\kappa}x_{s-j -1}, 2^{-\kappa}x_{s-j}])}.
	\end{equation}
	
	By combining  \eqref{f-P_{rho,m}f,x_ka}, \eqref{f-Q_{rho,m}f,x_kc}, \eqref{P(f-T)}, \eqref{|R_{rho,m}(f -  T_rf)|(2)}, \eqref{|RQ_{rho,m}(f -  T_rf)|} and  \eqref{|Q_{rho,m}(f -  T_rf)|}, we have
	\begin{equation} \nonumber
		\|f - P_{\rho,m} f\|_{L_{q,w}([x_k,x_{k+1}])}
		\ll 	m^{-r'_{\lambda,p,q}} \brac{A^T_k + A^R_k + A^Q_k + A^{RQ}_k},  
		\ \ \forall x \in [x_k,x_{k+1}], \ \forall k \in \ZZ,
	\end{equation}
	where
	\begin{equation}\nonumber
		A^T_k :=	\big\|f^{(r)}\big\|_{L_{p,w}([x_k, x_{k+1}])}, \ \ 
		A^R_k:= \sum_{(s,i) \in J_k^R}	\big\|f^{(r)}\big\|_{L_{p,w}([2^{-\kappa}x_{s-1}, 2^{-\kappa}x_s])},
	\end{equation}	
	\begin{equation}\nonumber
		A^Q_k := \sum_{(s,i,j) \in J_k^Q} \big\|f^{(r)}\big\|_{L_{w,p}([x_{s-j-1}, x_{s-j}])}, \ \ 
		A^{RQ}_k:= \sum_{(s,i,j) \in J_k^{RQ}}	\big\|f^{(r)}\big\|_{L_{p,w}([2^{-\kappa}x_{s-j -1}, 2^{-\kappa}x_{s-j}])}.
	\end{equation}			
	Based on this inequality  by arguments and estimations similar to 
	\eqref{|f - Q_{rho,m}f|}--\eqref{f-Q_{rho,m}f,a_m(b)} in the proof of \eqref{f-Q_{rho,m}f,T_{rho,m}} we prove  \eqref{f-P_{rho,m}f,T_{rho,m}}.
	The theorem has been proven.
	\hfill	
\end{proof}	

\subsection{Proof of Theorem \ref{thm:MarcinkiewiczInequality}} 
\label{Proof of Theorem ref{thm:MarcinkiewiczInequality}}

 \begin{proof}	
 	Fix a positive number $\rho$ satisfying \eqref{rho<max}.
 	We first prove the norm equivalence 
 	\begin{equation}\label{MarcinkiewiczInequality1} 
 		\|\varphi\|_{L_{p,w}(\RR)} 
 		\asymp
 		m^{(1/\lambda - 1)/p}\|(\varphi(x_s))\|_{p,w,m} 
 		\ \ 	\forall \varphi \in S_{\rho,m},   \ \forall m  \ge \ell. 
 	\end{equation}
 	Due to \eqref{supp varphi}, to prove \eqref{MarcinkiewiczInequality1} it is sufficient to show that there exists  a number 
 	$\rho:= \rho(a,\lambda,\ell,j_0)$  with $0 < \rho < 1$  such that for the first term in the right-hand side of \eqref{f-Q},
 	\begin{equation}\label{varphi><}
 		\|\varphi\|_{L_{q,w}([-\rho a_m,\rho a_m])}
 		\asymp
 		m^{(1/\lambda - 1)/p}
 		\|(\varphi(x_s)\|_{p,w,m} \ \ \  
 		\forall \varphi \in S_{\rho,m},   \ \forall m  \ge \ell. 
 	\end{equation}
 	Let $ \varphi \in S_{\rho,m}$. 	We have
 	\begin{equation*}
 		\|\varphi\|_{L_{p,w}([-\rho a_m,\rho a_m])}^p
 		=
 		\sum_{k=-m}^{m  - 1} \|\varphi\|_{L_{p,w}([x_k,x_{k+1}])}^p.
 	\end{equation*}
 	By  \eqref{Q_{rho,m}:=} and   \eqref{varphi=(1)} for  $x \in [x_k, x_{k+1}]$,
 	\begin{equation} \nonumber
 		\varphi(x)  \ = \ 
 		\sum_{s =k-\ell+1}^{k+\ell} \sum_{|j| \le j_0}  
 		\sum_{i =0}^{2\ell} c_{i,j} h_m^{1-2\ell}\varphi(x_{s - j}) (x - x_{s + i  - \ell})_+^{2\ell - 1}, 
 	\end{equation}
 	where $c_{i,j}$ is as in \eqref{c_{i,j}:=}.
 	We rewrite the last equality in a more compact form as
 	\begin{equation} \nonumber
 		\varphi(x)  \ = \ 
 		\sum_{(s,i,j) \in J_k} c_{i,j} 	F_{\xi,\eta}\varphi(x) \ \
 		\forall x \in [x_k, x_{k+1}], 
 	\end{equation}
 	where  $J_k^Q$ is as in \eqref{J_k^Q}, $\xi,\eta$ as in \eqref{eta,xi}
 	and
 	$F_{\xi,\eta}$ as in \eqref{F_{xi,eta}}.
 	With the chosen number 
 	$\rho$ satisfying  \eqref{rho<max},
 	and $\eta,\xi$ as in \eqref{eta,xi}, we get by \eqref{(x - x_s)_+},
 	\begin{equation}\nonumber
 		|F_{\xi,\eta}\varphi(x)|w(x)
 		\le 
 		|\varphi(x_\eta)|w(x_\eta) 
 		\ \
 		\forall x \in [x_k, x_{k+1}], \ \ \forall (s,i,j) \in J_k^Q.
 	\end{equation}
 	By applying  the norm $\|\cdot\|_{L_{p,w}([x_k,x_{k+1}])}$  to the left-hand side we get
 	\begin{equation}\nonumber
 		\|F_{\xi,\eta}\varphi\|_{L_{q,w}([x_k,x_{k+1}])}^p 
 		\le 
 		h_m	|\varphi(x_\eta)w(x_\eta)|^p 
 		\ \
 		\forall (s,i,j) \in J_k^Q.
 	\end{equation} 	
 	From the last inequality, \eqref{f-Q_{rho,m}f,x_ka} and \eqref{f-Q_{rho,m}f,x_kc} it follows that
 	\begin{equation} \nonumber
 		\|\varphi\|_{L_{p,w}( [x_k, x_{k+1}])}^p 	
 		\ \ll \ 
 		m^{1/\lambda - 1}\sum_{(s,i,j) \in J_k^Q} |\varphi(x_{s-j})w(x_{s-j})|^p. 	
 	\end{equation}
 	Since $\varphi(x_s)=0$ for $|s|> m$, we obtain	
 	\begin{equation}\nonumber
 		\begin{aligned}			
 			\|\varphi\|_{L_{p,w}([-\rho a_m,\rho a_m])}^p
 			\ &\ll \ 
 			m^{1/\lambda - 1}	\sum_{k=-m }^{m  - 1}
 			\sum_{(s,i,j) \in J_k^Q} |\varphi(x_{s-j})w(x_{s-j})|^p
 			\\ &\ll \ 
 			m^{1/\lambda - 1}	\sum_{|s| \le m}
 			|\varphi(x_s)w(x_s)|^p	 
 		\end{aligned}
 	\end{equation}
 	This proves the inequality
 	\begin{equation}\label{varphi<}
 		\|\varphi\|_{L_{q,w}([-\rho a_m,\rho a_m])}
 		\ll
 		m^{(1/\lambda - 1)/p}\|(\varphi(x_s))\|_{p,w,m}  \ \  
 		\forall \varphi \in S_{\rho,m},   \ \forall m  \ge \ell. 
 	\end{equation}
 	
 	Let us prove the inverse inequality. Let $|s| \le m$. We assume $s \ge 1$.  The case $s < 1$ can be treated analogously with a modification.  From \cite[(2.14), Chapter 4]{DeLo93B} it follows that
 	\begin{equation}\nonumber
 		|\varphi(x_s)|^p
 		\le 
 		\|\varphi\|_{L_\infty([x_{s-1},x_s])}^p 
 		\ll h_m^{-1} \|\varphi\|_{L_p([x_{s-1},x_s])}^p = h_m^{-1} \int_{x_{s-1}}^{x_s} |\varphi(x)|^p \rd x. 		 	
 	\end{equation}
 	Hence,
 	\begin{equation}\label{|varphi(x_s)w(x_s)|^p}
 		|\varphi(x_s)w(x_s)|^p
 		\ll   h_m^{-1}\int_{x_{s-1}}^{x_s} |\varphi(x)w(x)|^p \rd x =\|\varphi\|_{L_{p,w}([x_{s-1},x_s])}^p,		 	
 	\end{equation}
 	and, consequently,		 
 	\begin{equation}\nonumber
 		m^{1/\lambda - 1}	\sum_{|s| \le m}|\varphi(x_s)w(x_s)|^p
 		\ll   \sum_{|s| \le m} \int_{x_{s-1}}^{x_s} |\varphi(x)w(x)|^p \rd x 
 		=\|\varphi\|_{L_{p,w}([-\rho a_m,\rho a_m])}^p,		 	
 	\end{equation}
 	which establishes the inverse inequality in \eqref{varphi><}. The norm equivalence \eqref{MarcinkiewiczInequality1} has been proven.
 	
 	We now prove the second norm equivalence in \eqref{MarcinkiewiczInequality}:
 	\begin{equation}\label{MarcinkiewiczInequality2} 
 		\|\varphi\|_{L_{p,w}(\RR)} 
 		\asymp
 		m^{(1/\lambda - 1)/p}\|(b_s(\varphi)\|_{p,w, m-\ell}
 		\ \ 	\forall \varphi \in S_{\rho,m},   \ \forall m  \ge \ell. 
 	\end{equation}
 	Due to \eqref{supp varphi}, to prove \eqref{MarcinkiewiczInequality2}  it is sufficient to show that there exists  a number 
 	$\rho:= \rho(a,\lambda,\ell,j_0)$  with $0 < \rho < 1$  such that for the first term in the right-hand side of \eqref{f-Q},
 	\begin{equation}\label{varphi><2}
 		\|\varphi\|_{L_{q,w}([-\rho a_m,\rho a_m])}
 		\asymp
 		m^{(1/\lambda - 1)/p}
 		\|(b_s(\varphi))\|_{p,w, m-\ell} \ \ \  
 		\forall \varphi \in S_{\rho,m},   \ \forall m  \ge \ell. 
 	\end{equation}
 	Let $ \varphi \in S_{\rho,m}$. 	
 	By  \eqref{Q_{rho,m}:=}  for  $x \in [x_k, x_{k+1}]$,
 	\begin{equation} \nonumber
 		\varphi(x)  \ = \ 
 		\sum_{s =k-\ell+1}^{k+\ell}\sum_{i =0}^{2\ell}  
 		c_i b_s(\varphi) h_m^{1-2\ell}(x - x_{s + i  - \ell})_+^{2\ell - 1}, 
 	\end{equation}
 	where 
 	where 
 	\begin{equation} \label{c_{i,j}:=}
 		c_i:= \frac{1}{(2\ell - 1)!}(-1)^i \binom{2\ell}{i}.
 	\end{equation}
 	We rewrite the last equality in a more compact form as
 	\begin{equation} \nonumber
 		\varphi(x)  \ = \ 
 		\sum_{(s,i) \in J_k} c_i	F_{\xi,s}\varphi(x) \ \
 		\forall x \in [x_k, x_{k+1}], 
 	\end{equation}
 	where  $J_k:=\brab{(s,i): s =k-\ell+1,...,k+\ell, \ i =0,...,2\ell}$, $\xi = s + i - \ell$ 
 	and 
 	$$
 	F_{\xi,s}:= b_s(\varphi) h_m^{1-2\ell}(x - x_\xi)_+^{2\ell - 1}
 	$$ 
 	Similarly to \eqref{(x - x_s)_+} we can choose $0 < \rho < 1$ so that 
 	with $\eta,\xi$ as in \eqref{eta,xi}
 	by \eqref{(x - x_s)_+},
 	\begin{equation}\nonumber
 		|F_{\xi,s}\varphi(x)|w(x)
 		\le 
 		|b_s(\varphi)|w(x_s) 
 		\ \
 		\forall x \in [x_k, x_{k+1}], \ \ \forall (s,i) \in J_k.
 	\end{equation}
 	By applying  the norm $\|\cdot\|_{L_{p,w}([x_k,x_{k+1}])}$\eqref{f-Q_{rho,m}f,x_kb} to the left-hand side we get
 	\begin{equation}\nonumber
 		\|F_{\xi,s}\varphi\|_{L_{q,w}([x_k,x_{k+1}])}^p 
 		\ll
 		h_m	|b_s(\varphi)w(x_s)|^p 
 		\ \
 		\forall (s,i) \in J_k.
 	\end{equation} 	
 	Hence,  in the same way as the proof  of \eqref{varphi<} we  deduce the inequality
 	\begin{equation}\nonumber
 		\|\varphi\|_{L_{q,w}([-\rho a_m,\rho a_m])}
 		\ll
 		\|(b_s(\varphi))\|_{p,w, m-\ell} \ \ 
 		\forall \varphi \in S_{\rho,m},   \ \forall m  \ge \ell. 
 	\end{equation}
 	
 	Let us prove the inverse inequality. Let $|s| \le m - \ell$. We assume $s \ge 1$.  The case $s < 1$ can be treated analogously with a modification.  From \cite[Lemma 4.1, Chapter 4]{DeLo93B} it follows that
 	\begin{equation} \nonumber
 		|b_s(\varphi)|^p
 		\le 
 		\|\varphi\|_{L_\infty([x_{s-1},x_s])}^p 
 		\ll h_m^{-1} \|\varphi\|_{L_p([x_{s-1},x_s])}^p = h_m^{-1} \int_{x_{s-1}}^{x_s} |\varphi(x)|^p \rd x. 		 	
 	\end{equation}
 	Hence,
 	\begin{equation}\nonumber
 		|b_s(\varphi)w(x_s)|^p
 		\ll   h_m^{-1}\int_{x_{s-1}}^{x_s} |\varphi(x)w(x)|^p \rd x =\|\varphi\|_{L_{p,w}([x_{s-1},x_s])}^p,		 	
 	\end{equation}
 	and, consequently,		 
 	\begin{equation}\nonumber
 		m^{1/\lambda - 1}	\sum_{|s| \le m - \ell}|b_s(\varphi)w(x_s)|^p
 		\ll   \sum_{|s| \le m - \ell} \int_{x_{s-1}}^{x_s} |\varphi(x)w(x)|^p \rd x 
 		\le \|\varphi\|_{L_{p,w}([-\rho a_m,\rho a_m])}^p,		 	
 	\end{equation}
 	which establishes the inverse inequality in \eqref{varphi><2}. The norm equivalence \eqref{MarcinkiewiczInequality2} has been proven. The proof of the theorem is complete.
 	\hfill	
 \end{proof}

 \subsection{Proof of Theorem \ref{thm:BernsteinInequality}} 
 \label{Proof of Theorem ref{thm:BernsteinInequality}}
 
 \begin{proof}
 Fix a positive number $\rho$ satisfying \eqref{rho<max}.
 	Let 	
 	$ \varphi \in S_{\rho,m}$. 
 	Due to \eqref{supp varphi}, to prove \eqref{varphi^{(r)}} it is sufficient to show that there exists $0 < \rho < 1$ such that
 	\begin{equation}\label{varphi^{(r)}<}
 		\|\varphi^{(r)}\|_{L_{q,w}([-\rho a_m,\rho a_m])}
 		\ll
 		m^{ r_\lambda} \|f \|_{L_{p,w}(\RR)}.  
 	\end{equation}
 	We have
 	\begin{equation*}
 		\|\varphi^{(r)}\|_{L_{p,w}([-\rho a_m,\rho a_m])}^p
 		=
 		\sum_{k=-m + \ell}^{m - \ell - 1} \|\varphi^{(r)}\|_{L_{p,w}([x_k,x_{k+1}])}^p.
 	\end{equation*}
 	By  \eqref{Q_{rho,m}:=}  for  $x \in [x_k, x_{k+1}]$,
 	\begin{equation} \nonumber
 		\varphi^{(r)}(x)  \ = \ 
 		\sum_{s =k-\ell+1}^{k+\ell} \sum_{|j| \le j_0}  
 		\sum_{i =0}^{2\ell} c_{i,j} h_m^{1-2\ell}\varphi(x_{s - j}) (x - x_{s + i  - \ell})_+^{2\ell - 1 - r}, 
 	\end{equation}
 	where $c_{i,j}$ is as in \eqref{c_{i,j}:=}.
 	
 	We rewrite the last equality in a more compact form as
 	\begin{equation} \label{varphi^{(r)}(x)=}
 		\varphi^{(r)}(x)  \ = \ 
 		\sum_{(s,i,j) \in J_k^Q} c_{i,j} 	F_{\xi,\eta}\varphi(x) \ \
 		\forall x \in [x_k, x_{k+1}], 
 	\end{equation}
 	where  $J_k^Q$ is as in \eqref{J_k^Q}, $\xi,\eta$ as in \eqref{eta,xi}
 	and 
 	\begin{equation}\nonumber
 		F_{\xi,\eta}\varphi(x):=	\varphi(x_\eta) h_m^{1-2\ell+r}(x - x_\xi)_+^{2\ell - 1-r}.
 	\end{equation}	
 	With the chosen number 
 	$\rho$ satisfying  \eqref{rho<max},
 	and $\eta,\xi$ as in \eqref{eta,xi}, we get by \eqref{(x - x_s)_+},
 	\begin{equation}\nonumber
 		h_m^{1 - 2\ell + r} (x - x_\xi)_+^{2\ell - 1 - r}w(x)
 		\ll 
 		w(x_\eta)  \ \
 		\forall x \in [x_k, x_{k+1}], \ \ (s,i,j) \in J_k^Q.
 	\end{equation}
 	Hence, with $\eta,\xi$ as in \eqref{eta,xi}
 	we have
 	\begin{equation}\nonumber
 		|F_{\xi,\eta}\varphi(x)|w(x)
 		\le 
 		h_m^{-r}	|\varphi(x_\eta)|w(x_\eta) 
 		\ \
 		\forall x \in [x_k, x_{k+1}], \ \ \forall (s,i,j) \in J_k^Q.
 	\end{equation}
 	By \eqref{|varphi(x_s)w(x_s)|^p}
 	\begin{equation}\label{|varphi(x_s)w(x_s)|^p}
 		|\varphi(x_\eta)w(x_\eta)|^p
 		\ll   h_m^{-1}\int_{x_{\eta-1}}^{x_\eta} |\varphi(x)w(x)|^p \rd x =\|\varphi\|_{L_{p,w}([x_{\eta-1},x_\eta])}^p.		 	
 	\end{equation}
 	By applying  the norm $\|\cdot\|_{L_{p,w}([x_k,x_{k+1}])}$\eqref{f-Q_{rho,m}f,x_kb} to both the  sides we get 
 	\begin{equation}\nonumber
 		\|F_{\xi,\eta}\varphi\|_{L_{p,w}([x_k,x_{k+1}])}^p 
 		\le 
 		m^{pr_\lambda} \big\|\varphi\big\|_{L_{p,w}([x_{\eta - 1}, x_\eta])}^p\ \
 		\forall (s,i,j) \in J_k^Q,
 	\end{equation}
 	which together with \eqref{varphi^{(r)}(x)=} implies 
 	\begin{equation}\nonumber
 		\|\varphi^{(r)}\|_{L_{p,w}([x_k,x_{k+1}])}^p 
 		\ll
 		m^{pr_\lambda}	\sum_{(s,i,j) \in J_k^Q}  \big\|\varphi\big\|_{L_{w,p}([x_{\eta-1}, x_\eta])}^p.
 	\end{equation}
 	Hence, similarly to \eqref{|f - Q_{rho,m}f|} and \eqref{f-Q_{rho,m}f,a_m(a)} we derive
 	\begin{equation}\nonumber
 		\|\varphi^{(r)}\|_{L_{p,w}([-\rho a_m,\rho a_m])}^p 
 		\ll
 		m^{pr_\lambda}	\sum_{k=-m  - j_0}^{m + j_0 - 1}\sum_{(s,i,j) \in J_k^Q}  \big\|\varphi\big\|_{L_{w,p}([x_{\eta-1}, x_\eta])}^p
 		\ll m^{pr_\lambda}	\|\varphi\|_{L_{p,w}(\RR)}^p,
 	\end{equation}
 	which proves \eqref{varphi^{(r)}<}.
 	\hfill	
 \end{proof}

\bibliographystyle{abbrv}
\bibliography{WeightedBsplineApprox}

\begin{thebibliography}{10}

\bibitem{BISS2005}
D.~Barrera, M.~Ib{\`a}{\~n}ez, P.~Sablonni{\`e}re, and D.~Sbibih.
\newblock {Near minimally normed spline quasi-interpolants on uniform
  partitions}.
\newblock {\em J. Comput. Appl. Math.}, 181:211--233, 2005.

\bibitem{Bon1984}
S.~Bonan.
\newblock {Applications of G. Freud's theory, I}.
\newblock {\em Approximation Theory, IV (C. K. Chui et al., Eds.), Acad.
  Press}, pages pp. 347--351, 1984.

\bibitem{BP1984}
S.~Bonan and P.~Nevai.
\newblock {Orthogonal polynomials and their derivative}.
\newblock {\em J. Approx. Theory}, 40:134--147, 1984.

\bibitem{Chui92}
C.~K. Chui.
\newblock {\em {An Introduction to Wavelets}}.
\newblock Academic Press, 1992.

\bibitem{Dung2018}
D.~{D\~ ung}.
\newblock {B-spline quasi-interpolation sampling representation and sampling
  recovery in Sobolev spaces of mixed smoothness}.
\newblock {\em Acta Math. Vietnamica}, 43:83--110, 2018.

\bibitem{DD2023}
D.~{D\~ ung}.
\newblock Numerical weighted integration of functions having mixed smoothness.
\newblock {\em J. Complexity}, 78:101757, 2023.

\bibitem{DD2023-survey}
D.~{D\~ ung}.
\newblock Sparse-grid sampling recovery and numerical integration of functions
  having mixed smoothness.
\newblock {\em Acta Math. Vietnamica}, 49:377--426, 2024.

\bibitem{Dung11a}
D.~{D\~ung}.
\newblock {B-spline quasi-interpolant representations and sampling recovery of
  functions with mixed smoothness}.
\newblock {\em J. Complexity}, 27:541--567, 2011.

\bibitem{Dung16}
D.~{D\~ung}.
\newblock {Sampling and cubature on sparse grids based on a B-spline
  quasi-interpolation}.
\newblock {\em Found. Comp. Math.}, 16:1193--1240, 2016.

\bibitem{DD2024}
D.~{D\~ung}.
\newblock {Weighted sampling recovery of functions with mixed smoothness}.
\newblock {\em arXiv Preprint}, arXiv:2405.16400 [math.NA], 2024.

\bibitem{DK2022}
D.~{D\~ung} and V.~K. Nguyen.
\newblock {Optimal numerical integration and approximation of functions on
  $\mathbb{R}^d$ equipped with Gaussian measure}.
\newblock {\em IMA Journal of Numer. Anal.}, 44:1242--1267, 2024.

\bibitem{DTU18B}
D.~{D\~ung}, V.~N. Temlyakov, and T.~Ullrich.
\newblock {\em {Hyperbolic Cross Approximation}}.
\newblock Advanced Courses in Mathematics - CRM Barcelona,
  Birkh\"auser/Springer, 2018.

\bibitem{deBoor1977}
C.~de~Boor.
\newblock {Package for calculating with B-Splines}.
\newblock {\em SIAM J. Numer. Analysis}, 14:10.1137/0714026, 1977.

\bibitem{deBoor1978}
C.~de~Boor.
\newblock {\em {A Practical Guide to Splines}}.
\newblock Springer, 1978.

\bibitem{deBHR1993}
C.~de~Boor, K.~H{\"o}llig, and S.~Riemenschneider.
\newblock {\em {Box Spline}}.
\newblock Springer, Berlin, 1993.

\bibitem{DM2003}
B.~{Della Vecchia} and G.~Mastroianni.
\newblock {Gaussian rules on unbounded intervals}.
\newblock {\em J. Complexity}, 19:247--258, 2003.

\bibitem{DeLo93B}
R.~DeVore and G.~Lorentz.
\newblock {\em {Constructive Approximation}}.
\newblock Springer-Verlag, New York, 1993.

\bibitem{DILP18}
J.~Dick, C.~Irrgeher, G.~Leobacher, and F.~Pillichshammer.
\newblock {On the optimal order of integration in Hermite spaces with finite
  smoothness}.
\newblock {\em SIAM J. Numer. Anal.}, 56:684--707, 2018.

\bibitem{DKU22}
M.~Dolbeault, D.~Krieg, and M.~Ullrich.
\newblock {A sharp upper bound for sampling numbers in $L_2$}.
\newblock {\em Appl. Comput. Harmon. Anal.}, 63:113--134, 2023.

\bibitem{DD2025b}
D.~D{\~ung}.
\newblock {Optimal approximation and sampling recovery in measured-based
  function spaces}.
\newblock {\em VIASM Preprint}, ViAsM25.16, 2025.

\bibitem{DD2025a}
D.~D{\~ung}.
\newblock {Sampling reconstruction and integration of functions on
  $\mathbb{R}^d$ endowed with a measure}.
\newblock {\em VIASM Preprint}, ViAsM25.15, 2025.

\bibitem{DD2024a}
D.~D{\~ung}.
\newblock {Weighted hyperbolic cross polynomial approximation}.
\newblock {\em arXiv Preprint}, arXiv:2407.19442 [math.NA], 2024.

\bibitem{EG2022}
M.~Ehler and K.~Gr{\"o}chenig.
\newblock {Gauss quadrature for Freud weights, modulation spaces, and
  Marcinkiewicz-Zygmund inequalities}.
\newblock {\em arXiv Preprint}, arXiv:2208.01122 [math.NA], 2022.

\bibitem{Freud1976}
G.~Freud.
\newblock {On the coefficients in the recursion formulae of orthogonal
  polynomials}.
\newblock {\em Proc. R. Irish Acad., Sect. A}, 76:1--6, 1976.

\bibitem{GHHR2022}
M.~Gnewuch, M.~Hefter, A.~Hinrichs, and K.~Ritter.
\newblock {Countable tensor products of Hermite spaces and spaces of Gaussian
  kernels}.
\newblock {\em J. Complexity}, 71:101654, 2022.

\bibitem{GHRR2024}
M.~Gnewuch, A.~Hinrichs, K.~Ritter, and R.~R{\"u}ssmann.
\newblock {Infinite-dimensional integration and $L^2$-approximation on Hermite
  spaces}.
\newblock {\em J. Approx. Theory}, 300:106027, 2024.

\bibitem{GKS2024}
T.~Goda, Y.~Kazashi, and Y.~Suzuki.
\newblock {Randomizing the trapezoidal rule gives the optimal RMSE rate in
  Gaussian Sobolev spaces}.
\newblock {\em Math. Comp.}, 93:1655--1676, 2024.

\bibitem{IL2015}
C.~Irrgeher and G.~Leobacher.
\newblock {High-dimensional integration on the $\mathbb{R}^d$, weighted Hermite
  spaces, and orthogonal transforms}.
\newblock {\em J. Complexity}, 31:174--205, 2015.

\bibitem{JMN2021}
P.~Junghanns, G.~Mastroianni, and I.~Notarangelo.
\newblock {\em {Weighted Polynomial Approximation and Numerical Methods for
  Integral Equations}}.
\newblock Birkh\"auser, 2021.

\bibitem{KSG2023}
Y.~Kazashi, Y.~Suzuki, and T.~Goda.
\newblock {Sub-optimality of Gauss-Hermite quadrature and optimality of
  trapezoidal rule for functions with finite smoothness}.
\newblock {\em SIAM J. Numer. Analysis}, 61:1426--1448, 2023.

\bibitem{Lu07B}
D.~S. Lubinsky.
\newblock {A survey of weighted polynomial approximation with exponential
  weights}.
\newblock {\em Surveys in Approximation Theory}, 3:1--105, 2007.

\bibitem{Lubi82}
C.~Lubitz.
\newblock {\em {Weylzahlen von Diagonaloperatoren und Sobolev-Einbettungen}}.
\newblock PhD thesis, Bonner Math. Schriften, 144, Bonn, 1982.

\bibitem{MN2010}
G.~Mastroianni and I.~Notarangelo.
\newblock {A Lagrange-type projector on the real line}.
\newblock {\em Math. Comput.}, 79(269):327--352, 2010.

\bibitem{MO2004}
G.~Mastroianni and D.~Occorsio.
\newblock {Markov-Sonin Gaussian rule for singular functions}.
\newblock {\em J. Comput. Appl. Math.}, 169(1):197--212, 2004.

\bibitem{MV2007}
G.~Mastroianni and P.~{V\'ertesi}.
\newblock {Fourier sums and Lagrange interpolation on $(0,+\infty)$ and
  $(-\infty,+\infty)$}.
\newblock {\em {In: Frontiers in Interpolation and Approximation}}, {vol. 282
  of Pure Appl.Math. (Boca Raton) (Chapman and Hall/CRC, Boca Raton, FL,
  2007)}:307--344, 2007.

\bibitem{Mha1996B}
H.~N. Mhaskar.
\newblock {\em {Introduction to the Theory of Weighted Polynomial
  Approximation}}.
\newblock World Scientific, Singapore, 1996.

\bibitem{NoWo08}
E.~Novak and H.~Wo{\'z}niakowski.
\newblock {\em {Tractability of Multivariate Problems, Volume I: Linear
  Information}}.
\newblock EMS Tracts in Mathematics, Vol. 6, Eur. Math. Soc. Publ. House,
  Z\"urich, 2008.

\bibitem{NoWo10}
E.~Novak and H.~Wo{\'z}niakowski.
\newblock {\em {Tractability of Multivariate Problems, Volume II: Standard
  Information for Functionals}}.
\newblock EMS Tracts in Mathematics, Vol. 12, Eur. Math. Soc. Publ. House,
  Z\"urich, 2010.

\bibitem{OR2006}
D.~Occorsio and M.~Russo.
\newblock {The $L_p$-weighted Lagrange interpolation on Markov-Sonin zeros}.
\newblock {\em Acta Math. Hungar.}, 112(1-2):57--84, 2006.

\bibitem{SK2024}
Y.~Suzuki and T.~Karvonen.
\newblock {Construction of optimal algorithms for function approximation in
  Gaussian Sobolev spaces }.
\newblock {\em arXiv Preprint}, arXiv:2402.02917 [math.NA], 2024.

\bibitem{Sza1997}
J.~Szabados.
\newblock {Weighted Lagrange and Hermite-Fej{\'e}r interpolation on the real
  line}.
\newblock {\em J. Inequal.Appl.}, 1:99--123, 1997.

\bibitem{Tem18B}
V.~N. Temlyakov.
\newblock {\em {Multivariate Approximation}}.
\newblock Cambridge University Press, 2018.

\bibitem{Tri10B}
H.~Triebel.
\newblock {\em {Bases in Function Spaces, Sampling, Discrepancy, Numerical
  Integration}}.
\newblock European Math. Soc. Publishing House, Z\"urich, 2010.

\end{thebibliography}
\end{document}